\documentclass[12pt,a4paper]{amsart}

\usepackage{etoolbox,lastpage}
\usepackage{pifont}
\usepackage{xypic}
\usepackage[all]{xy}
\usepackage{tipa}
\usepackage{newlfont}

 \RequirePackage{amsmath}
\RequirePackage{amsthm}
\RequirePackage{amssymb}
\RequirePackage{mathrsfs}
\usepackage{bbding}
\usepackage{wasysym}
\usepackage{cleveref}

 \xyoption{2cell}
\xyoption{all}
\UseAllTwocells

 \newtheorem{theorem}{Theorem}[section]

\newtheorem{proposition}[theorem]{Proposition}
\newtheorem{lemma}[theorem]{Lemma}
\newtheorem{corollary}[theorem]{Corollary}

\theoremstyle{definition}
\newtheorem{definition}[theorem]{Definition}
\theoremstyle{definition}
\newtheorem{remark}[theorem]{Remark}
\theoremstyle{definition}

\theoremstyle{definition}
\newtheorem{example}[theorem]{Example}
\newtheorem{notation}[theorem]{Notation}

\def\hs#1{\hskip#1mm}%
\def\su{{\subseteq \hs1}}
\def\cc{{\mathcal{C}}}
\def\ce{{\mathcal{E}}}
\def\c{{\mathcal{C}}}

\def\cm{{\mathcal{M}}}
\def\ca{{\mathcal{A}}}
\def\cd{{\mathcal{D}}}
\def\ch{{\mathcal{H}}}
\def\cx{{\mathcal{X}}}
\def\ck{{\mathcal{K}}}
\def\cn{{\mathcal{N}}}

\def\mc#1{\mathcal{#1}}
\def\mb#1{\mathbf{#1}}

\def\la{{\langle}}
\def\ra{{\rangle}}

\DeclareMathOperator{\iso}{Iso}
\DeclareMathOperator{\mon}{Mon}

\def\itemi{\item[{\rm(i)}]}
\def\itemii{\item[{\rm(ii)}]}
\def\itemiii{\item[{\rm(iii)}]}
\def\itemiv{\item[{\rm(iv)}]}

\def\itema{\item[{\rm(a)}]}
\def\itemb{\item[{\rm(b)}]}
\def\itemc{\item[{\rm(c)}]}
\def\itemd{\item[{\rm(d)}]}
\def\sqdg#1#2#3#4#5#6#7#8#9{$\xymatrix{ #1 \ar[rr]^{#5} \ar@{} \ar[d]_{#6} & & #2 \ar[d]^{#7} \ar @{} [dll]
		|{#9}\\
		#3 \ar[rr]_{#8} & & #4
	}$}

\def\ult{\mbox{\begin{picture}(7,10)
		\put(1,0){\line(1,2){5}}
		\put(1,0){\line(0,1){10}
			\put(0,10){\line(1,0){5}}}
		\end{picture}
}}

 \def\ulto#1{\overset{\kern-.35em #1}{\ult}}

 \def\ulte{\overset{\kern-.35em \ce}{\ult}}

 \def\ultqd{\overset{\kern-.35em \ce^{QC}}{\ult}}

\newbox\abox\newdimen\heit

\begin{document}

 	\title[On General Closure Operators ...]{On
	General Closure Operators and Quasi Factorization Structures}
	
	\author[S. Sh. Mousavi, S. N. Hosseini, A. Ilaghi-Hosseini]{Seyed Shahin
	Mousavi, Seyed Naser Hosseini and Azadeh
		Ilaghi-Hosseini}
	
	\address{Department of Pure Mathematics, Shahid Bahonar University of
	Kerman, Kerman, Iran.}
	
	\email{smousavi@uk.ac.ir}
	
	\address{Department of Pure Mathematics, Shahid Bahonar University of
	Kerman, Kerman, Iran.}
	
	\email{nhoseini@uk.ac.ir}
	
	\address{Department of Pure Mathematics, Shahid Bahonar University of
	Kerman, Kerman, Iran.}
	
	\email{a.ilaghi@math.uk.ac.ir}
	
	
	
\begin{abstract}

 	In this article the notions of quasi mono (epi) as a generalization of mono (epi), (quasi weakly hereditary) general closure operator $\mb{C}$ on a category $\cx$ with respect to a class $\cm$ of morphisms, and quasi factorization structures in a category $\cx$ are introduced. It is shown that under certain conditions,
if $(\ce, \cm)$ is a quasi factorization structure in $\cx$, then $\cx$ has a quasi right $\cm$-factorization structure and a quasi left $\ce$-factorization structure. It is also shown that for a quasi weakly hereditary and quasi idempotent QCD-closure operator with respect to a certain class $\cm$, every quasi factorization structure $(\ce, \cm)$ yields a quasi factorization structure relative to the given closure operator; and that for a closure operator with respect to a certain class $\cm$, if the pair of classes of quasi dense and quasi closed morphisms forms a quasi factorization structure, then the closure operator is both quasi weakly hereditary and quasi idempotent. Several illustrative examples are provided.

\end{abstract}

 \maketitle

 \section{Introduction and Preliminaries}
Closure operators have been around for almost one century in the context of categories of topological spaces and lattices. In \cite{s}, Salbany introduces a particular closure operator in the category of topological spaces. This idea was later transformed to an arbitrary category, which led to the general concept of categorical closure operators, \cite{dt}, \cite{dg}, \cite{dgt}. Weakly hereditary and idempotent closure operators play an important role, as they arise from factorization structures. In \cite{mh11}, quasi right factorization structures were introduced and their connection with closure operators was investigated, while quasi left factorization structures appear in \cite{hm}.

 There are many important structures that are not factorization structures nor even weak factorization structures; however they are quasi factorization structures, as introduced in this article. Many examples of such structures are provided and the connections between quasi right factorization structures, quasi left factorization structures, quasi factorization structures and closure operators are investigated.

 In section 2, to develop some theory related to closure operators in the more
general context of a quasi right factorization structure $\cm$ on a category
$\cx$, the notions of quasi mono and quasi epi are given. Then we will study
some preliminary results and we will provide some examples of these notions.
The strong point of these examples is to provide epimorphisms which are quasi
mono and monomorphisms which are quasi epi. In section 3, the definition of a
general closure operator on a category $\cx$ with respect to the class $\cm$ of
morphisms is introduced, some related results and several examples are also
given. In section 4 after defining quasi weakly hereditary closure operator, we
prove that for a quasi idempotent closure operator we have a quasi right
factorization structure and for a quasi weakly hereditary closure operator
under some conditions we have a quasi left factorization structure. In section
5, for morphism classes $\ce$ and $\cm$, the notion of $(\ce, \cm)$-quasi
factorization structure is introduced and examples of quasi factorization
structures that are not weak factorization structures are furnished. It is
shown that if $(\ce, \cm)$ is a quasi factorization structure in $\cx$, then
$\cx$ has quasi right $\cm$-factorization structure provided that $\cm$ has
$\cx$-pullbacks and it has quasi left $\ce$-factorization structure provided
that $\cm \su \mon(\cx)$, the class of monos, and $\ce$ has $\cx$-pushouts.
It is also shown that for a quasi weakly hereditary and quasi idempotent QCD-closure operator with respect to a class $\cm$ that is contained in the class of quasi monos and is closed under composition, every quasi factorization structure $(\ce, \cm)$ yields a quasi factorization structure relative to the given closure operator. Finally it is proved that for a closure operator with respect to a class $\cm$ that is contained in the class of strongly quasi monos and is a codomain, if the pair of classes of quasi dense and quasi closed morphisms forms a quasi factorization structure, then the closure operator is both quasi weakly hereditary and quasi idempotent.

 To this end we will give some basic definitions and results which will be used in the following sections.

 \begin{definition}\cite{mh11}.\label{qrf}
	 Let $\cm$ be a class of morphisms in
	$\cx$ and for every object $ X$ of $ \cx$, $ \cm\slash X$ be the class of all morphisms with codomain $X$. We say that $\cx$ has {\it quasi right $\cm$-factorizations} or
	$\cm$ is a {\it quasi right factorization structure} in $\cx$, whenever
	for every morphisms $\hs{-1}\xymatrix{ Y \ar[r]^f & X}\hs{-1}$ in $\cx$, there
	exists $\xymatrix{ M \ar[r]^{m_{f}} & X}\hs{-1}\in \cm\slash X$ such that:
	
	(a) $f=m_fg$ for some $g$;
	
	(b) if there exists $m\in{\cm\slash X}$ such that $f=mg$ for some $g$, then $m_f=mh$ for some $h$.
	
	$m_f$ is called a quasi right part of $f$.
	
\end{definition}

 With $\langle m \rangle = \{ mh : h \in \cx \hskip3mm \hbox{and} \hskip3mm mh \hskip3mm \hbox{ is defined} \}$ denoting the sieve generated by $m$, see \cite{mm}, (a) is equivalent to:

 (a$'$) $\langle f \rangle \subseteq \langle m_{f} \rangle $;

 and (b) is equivalent to:

 (b$'$) if there exists $m\in{\cm\slash X}$ such that $\langle f \rangle \subseteq \langle m \rangle$,
then $\langle m_{f} \rangle\subseteq \langle m \rangle$.

 Note that right $\cm$-factorizations as defined in \cite{dt} are quasi right $\cm$-factorizations.

 \begin{lemma}\label{L:quasi right} \cite{mh11}. Suppose $\cx$ has quasi right 
 	$\cm$-factorizations. Let $f$ be a morphism in $\cx$ and $m_f$ be a quasi
	right part of $f$.
	\begin{itemize}
		\itema If $f\in \cm$, then $\langle m_f \rangle = \langle f \rangle$.
		
		\itemb $m$ is a quasi right part of $f$ if and only if $m\in\cm$ and
		$\langle m\rangle=\langle m_f\rangle$.
		
		\itemc If $\langle f \rangle \subseteq \langle g \rangle$, then $ \langle m_f \rangle \subseteq \langle m_g \rangle$.
		
		\itemd If $\langle g \rangle = \langle f \rangle$, then $m_f$ is a quasi right part of $g$.
	\end{itemize}
\end{lemma}

 The class of all isomorphisms in $\cx$ is denoted by $\text{Iso}(\cx)$.

 \begin{proposition}\label{p:composition of qrf by iso}
	Suppose $\cm$ is closed under composition with isomorphisms on the left,
	i.e., $m\in \cm$, $\alpha\in \iso(\cx)$, and $\alpha m$ defined, yields
	that $\alpha m \in \cm$. If $f$ is a morphism in $\cx$ and $m_f$ is a quasi
	right part of $f$, then $\alpha m_f$ is a quasi right part of $\alpha f$.
\end{proposition}

 \begin{proof}
	 This follows directly from the definition.
\end{proof}

 \begin{notation}
	For each composite $\hs{-1}\xymatrix{Z \ar[r]^g & X \ar[r]^f & Y}\hs{-1}$
	in $\cx$ we will denote by 	$\hs{-1}\xymatrix{f(g) : f(Z) \ar[r] &
	Y}\hs{-1}$ a chosen quasi right part of a quasi right $\cm$-factorization
	of the composite $fg$. Note that if $m'$ is another quasi right part of
	$fg$, by Lemma~\ref{L:quasi right}(b) we have $\la f(g) \ra = \la m' \ra$.
\end{notation}

\begin{remark}
	Suppose that $\cx$ has quasi right $\cm$-factorizations.
	
	(a) For each $\hs{-1}\xymatrix{M \ar[r]^m & X}\in \cm\hs{-1}$ we have $\la m(1_M) \ra = \la m \ra$.
	
	(b) For each $\hs{-1}\xymatrix{X \ar[r]^f & Y}\hs{-1}$, $\hs{-1}\xymatrix{T \ar[r]^g & Y}\hs{-1}$ and $\hs{-1}\xymatrix{Y \ar[r]^h & Z}\hs{-1}$ in $\cx$ if $\la f \ra \su \la g \ra$, then $\la h(f) \ra \su \la h(g) \ra$.
\end{remark}

 \begin{proposition}\label{p:composition by iso is not different}
	Suppose that $\cm$ is closed under composition with isomorphisms on the left. For each morphism $\hs{-1}\xymatrix{X \ar[r]^f & Y}\hs{-1}$ and isomorphism $\hs{-1}\xymatrix{Y \ar[r]^{\alpha} & Z}\hs{-1}$ in $\cx$ we have $\la \alpha(f) \ra = \la \alpha(f(1_X)) \ra$.
\end{proposition}

 \begin{proof}
	Obvious.
\end{proof}

The notion of a cosieve is dual to that of a sieve. A principal cosieve
generated by $f$ is denoted by $ \rangle f \langle $. Also
the notion of a quasi left $\ce$-factorization is dual of quasi right
$\cm$-factorization, see~\cite{hm}.

\section{Quasi mono and quasi epi}

 In this section, quasi monos and quasi epis as a generalization of monos and epis will be defined and some of their properties will be studied. Then we will provide some examples of these notions. The significant point of these examples is to provide epimorphisms which are quasi mono and monomorphisms which are quasi epi. Since these notions, especially ``quasi mono'', are used in the study of some kinds of general closure operators, some examples of quasi right $\cm$-factorization structures are given, in which the class $\cm$ is contained in the class of quasi monos.

 \begin{definition}\label{quasi mono}
	(a) A morphism $f$ in $\cx$ is called {\it quasi mono}, whenever for each morphism $a, b \in \cx$ if $f a = f b$, then $\la a \ra = \la b \ra$.
	The class of all quasi monos in $\cx$ is denoted by $QM(\cx)$.

	(b) A morphism $f$ is called {\it quasi epi}, whenever for each morphism $a, b \in \cx$ if $a f = b f$, then $\la a \ra = \la b \ra$.
The class of all quasi epis in $\cx$ is denoted by $QE(\cx)$.

\end{definition}

 \begin{proposition}\label{p:quasi mono}
	We have the following:
	\begin{itemize}
		\itema $f$ is a quasi mono if and only if for all morphisms $u$ and $v$ with the same codomain if $\la fu \ra \su \la fv \ra$, then $\la u \ra \su \la v \ra$.
		
		\itemb if $gf$ is a quasi mono, then $f$ is a quasi mono.
		
		\itemc if $f$ and $g$ are quasi monos and $gf$ is defined, then $gf$ is a quasi mono.
		
		\itemd if $\hs{-1}\xymatrix{A \ar[r]^{f} & C}\hs{-1}$ is a quasi mono, and $\hs{-1}\xymatrix{B \ar[r]^{g} & C}\hs{-1}$ is a mono and the diagram
		$$\xymatrix{A \times_{_{C}} B \ar[d]_{\pi_1} \ar @{} [dr] |{p.b} \ar[r]^{\hs{4}\pi_2} & B \ar[d]^{g}\\
			A \ar[r]_{f} & C}$$
		is a pullback in $\cx$, then $\pi_2$ is a quasi mono.
	\end{itemize} 
\end{proposition}
\begin{proof}
	The proof is straightforward.
\end{proof}

 \begin{proposition}
	We have the following:
	\begin{itemize}
		\itema if $gf$ is a quasi epi, then $g$ is a quasi epi.
		
		\itemb if $f$ is an epi and $g$ is a quasi epi and $gf$ is defined, then $gf$ is a quasi epi.
	\end{itemize} 
\end{proposition}
\begin{proof}
	The proof is straightforward.
\end{proof}

 \begin{lemma}\label{l:quasi mono split epi is iso} We have the following:
 	\begin{itemize}
 		\itema if $f$ is a quasi mono and a split epi, then $f$ is an isomorphism.
 		
 		\itemb if $f$ is a quasi epi and a split mono, then $f$ is an isomorphism.
 	\end{itemize} 
\end{lemma}

 \begin{proof}
	(a) Since $\hs{-1}\xymatrix{X \ar[r]^{f} & Y}\hs{-1}$ is a split epi, we have $\hs{-1}\xymatrix{Y \ar[r]^{s} & X}\hs{-1}$ such that $f s = 1_Y$. Thus $fsf = f$, and so $\la sf \ra = \la 1_X \ra$, because $f$ is a quasi mono. Therefore there exists a morphism $\hs{-1}\xymatrix{X \ar[r]^{t} & X}\hs{-1}$ such that $sft = 1_X$. Hence $s$ and thus $f$ is an isomorphism.
	
	(b) Similar to (a).
\end{proof}

 \begin{corollary} We have the following:
 	\begin{itemize}
 		\itema if the class of quasi monos in $\cx$ is pullback stable, then quasi monos are monos.
 		
 		\itemb if the class of quasi epis is pushout stable, then quasi epis are epis.
 	\end{itemize}
\end{corollary}

 \begin{proof}
	(a) Suppose that a quasi mono $\hs{-1}\xymatrix{X \ar[r]^{f} & Y}\hs{-1}$ is given. Consider the following pullback diagram:
	$$\xymatrix{ X \ar@{} [ddr] |{_{///}} \ar@{} [drr] |{_{///}}
		\ar@/_2pc/[ddr]_{1_X} \ar@/^2pc/[drr]^{1_X}
		\ar@{.>}[dr]|{^{\exists ! g}} & \\
		& X\times_{_{
		Y}} X \ar [d]_{\pi_1} \ar@{} [dr] |{p.b.} \ar[r]^{\pi_2}
		& X \ar [d]^{f} & \\
		& X \ar[r]_{f} & Y }$$	
	There exists a unique morphism$\xymatrix{g:X \ar[r] &  X\times_{_{Y}} X}$ in $\cx$ such that $\pi_{1} g = 1_X$, $\pi_{2} g = 1_X$. Thus $\pi_2$ and $\pi_1$ are quasi monos and split epis, and so by Lemma~\ref{l:quasi mono split epi is iso}(a) they are isomorphisms. Therefore $f$ is a mono.
	
	(b) Similar to (a).
\end{proof}

 \begin{remark}
	In the categories, $\mb{Set}$ of sets, $\mb{Top}$, of topological spaces and $R$-$\mb{Mod}$, of left $R$-modules, quasi monos (quasi epis) are monos (epis).
\end{remark}

 To give examples of quasi monos that are not monos, we need the following definitions.

 Recall that, \cite[p.72]{af},  a submodule $K$ of $M$ is

 (i) superfluous in $M$, abbreviated $K \ll M$, in case for every submodule $L \leq M$, $K + L = M$ implies $L = M$.

 (ii) is essential in $M$, abbreviated $K \trianglelefteq M$, in case for every submodule $L \leq M$, $K \cap L = 0$ implies $L = 0$. It is easy to see that $K$ is essential in $M$ if and only if, for every nonzero $m\in M$, there is $r \in R$ with $rm \in K$ and $rm \neq 0$.

 For a module $M$ the socle of $M$ ($= \text{Soc}(M)$) is defined as the sum of all simple (minimal) submodules of $M$, (see \cite[p.174]{w}).

 An $R$-module $M$ is called:

 (i) uniform if every non-zero submodule of $M$ is essential in $M$ (see \cite[19.9]{w}).

 (ii) quasi-projective in case $M$ is $M$-projective, (see \cite{fh}).

 (iii) cocyclic if there is an $m_0 \in M$ with the property: every morphism $\xymatrix{h:M \ar[r]& N}$with $m_0 \notin K_h$ is monic, where $K_h$ is the kernel of $h$, (see \cite[p.115]{w}).

 Let $R$ be a domain and let $M$ be an $R$-module. $M$ is called torsion-free, whenever there are no nonzero $x \in M$ and $r \in R$ with $r x = 0$. Note that projective modules are torsion free (see \cite[ Proposition 1.1]{ce}).

 \begin{example}\label{here}
	Let $\mathbb{QP}$ be the full subcategory of $R$-$\mb{Mod}$ whose objects are quasi-projective modules and let $R$ be a domain.
	
	 (a) Let $P$ be a uniform projective module. By \cite[21.1]{w}, $\text{Soc}(P)$ is the intersection of essential submodules of $P$ and since $P$ is uniform, every non-zero submodule of $P$ contains $\text{Soc}(P)$ (and hence it is minimum proper submodule of $P$). Consider the following commutative diagram
	$$\xymatrix{P \ar@{.{>}} [dr]_f \ar[r]^{\hs{-7}p} & P/\text{Soc}(P)\\
		& R^{(\Lambda)} \ar@{-{>>}} [u]_q}
	$$
	where $P/\text{Soc}(P)$ is homomorphic image of a free $R$-module $R^{(\Lambda)}$ for some set $\Lambda$, $p$ is the canonical projection and $q f = p$. Therefore $f$ is a non-zero map. We show that $f$ is quasi mono in $\mathbb{QP}$. To this end let the diagram,
	$$ \xymatrix{N \ar@{->}^{g} @<3pt> [rr]
		\ar@{->}_{h} @<-3pt> [rr] & & P \ar[rr]^{f} & & R^{(\Lambda)}}$$
	in $\mathbb{QP}$ be given such that $f g = f h$. So we have,
	
	\begin{equation}\label{eq 1}
	K_f + \text{Im}g = K_f + \text{Im}h
	\end{equation}
	If $\text{Im}h = 0$, then $\text{Im}g = 0$. Because otherwise, if $\text{Im}g \neq 0$, then $\text{Im}g \trianglelefteq P$. Since $f \neq 0$ there exists $0 \neq p_0 \in P$ such that $f(p_0) \neq 0$. This implies that there exists $r \in R$ such that $0 \neq r p_0 \in \text{Im}g$. Thus $r p_0 = g(t)$ for some $t \in N$ and hence $r f(p_0) = f(g(t)) = f(h(t)) = 0$. This is a contradiction, because $R^{(\Lambda)}$ is torsion-free. So $\text{Im}g = 0$. Thus we have two cases:
	
	(i) $\text{Im}g = \text{Im}h = 0$, and so $\la g \ra = \la h \ra = \{0\}$.
	
	(ii) $\text{Im}g \neq 0$ and $\text{Im}h \neq 0$. Thus equality (1) implies that $\text{Im}g = \text{Im}h$. Since $N$ is quasi-projective, there exist morphisms $\hs{-1}\xymatrix{\alpha , \beta : N \ar[r] & N}\hs{-1}$ such that the following diagram commutes
	$$\xymatrix{N \ar@ {-{>>}}[dr]_{\bar{g}} \ar@<1ex> @{.>} [rr]^{\hs{1} \alpha} & & N \ar@ {-{>>}}[dl]^{\bar{h}} \ar@<1ex> @{.>} [ll]^{\hs{1}\beta}\\
		& \text{Im}g = \text{Im}h & }$$
	where for each $x \in N$, $\bar{g}(x) = g(x)$ and $\bar{h}(x) = h(x)$. Therefore $\la g \ra = \la h \ra$ and hence $p$ is quasi mono.

(b) Let $P$ be a cocyclic projective module and $L$ be the intersection of all non-zero submodules of $P$, hence by \cite[14.8(c)]{w}, $L$ is a non-zero minimum submodule of $P$. Since every non-zero projective module contains a maximal submodule (see \cite[Proposition 17.14]{af}), $L$ is a proper submodule of $P$. Consider the following commutative diagram
$$\xymatrix{P \ar@{.{>}} [dr]_f \ar[r]^{\hs{-3}p} & P/L\\
	& R^{(\Lambda)} \ar@{-{>>}} [u]_q}
$$
where $P/L$ is homomorphic image of a free $R$-module $R^{(\Lambda)}$ for some set $\Lambda$, $p$ is the canonical projection and $q f = p$. As in (a) we can see that $f$ is a quasi mono in $\mathbb{QP}$.
\end{example}

Recall that, \cite[p.348]{w}, for a submodule $U$ of a left $R$-module $M$, a submodule $V \leq M$ is
called a supplement or addition complement of $U$ in $M$ if $V$ is a minimal
element in the set of submodules $L \leq M$ with $U + L = M$. A module is called supplemented if every submodule has supplements, (see \cite[p.349]{w}).

 A projective left (right) module over a ring $R$ will be called left hereditary in case every left submodule is projective, (see \cite{hi}).

 \begin{example}\label{supp}
		Let $M$ be a supplemented left hereditary module and $0 \neq K \ll M$. Let$\xymatrix{p: M \ar[r] & M/K}$ be the canonical projection. We show that $p$ is quasi mono in $R$-$\mb{Mod}$. To this end let the diagram,
		$$ \xymatrix{N \ar@{->}^{g} @<3pt> [rr]
			\ar@{->}_{h} @<-3pt> [rr] & & M \ar[rr]^{p} & & M/K}$$
		 such that $p g = p h$ be given. Thus $K + \text{Im}g = K + \text{Im}h$. Since $K + \text{Im}g$ has a supplement in $M$, by \cite[41.1(5)]{w} we have $K = K \cap (K + \text{Im}g) \ll (K + \text{Im}g)$. Therefore $\text{Im}h = K + \text{Im}g$ and hence $\text{Im}g \subseteq \text{Im}h$. Similarly $\text{Im}h \subseteq \text{Im}g$, so $\text{Im}g = \text{Im}h$. Since $M$ is hereditary, $\text{Im}g$ is a projective submodule of $M$. So there exist$\xymatrix{\iota_1: \text{Im}g \ar[r] & N}$and$\xymatrix{\iota_2: \text{Im}h \ar[r] & N}$ such that $g \iota_1 = 1_{\text{Im}g}$ and $h \iota_2 = 1_{\text{Im}h}$. Thus $g = h (\iota_2 g)$ and $h = g (\iota_1 h)$, hence $\la g \ra = \la h \ra$. Therefore $p$ is quasi mono.
		
\end{example}

\begin{example}\label{ex.m.1}
	Let $\mathbb{QP}$ be the category of quasi-projective left $R$-modules ($R$ need not be a domain).
	 Define the subcategory $\cc$ of $ \mathbb{QP} $ to have $\text{Obj}(\mathbb{QP})$ as objects and for all $ P,Q \in \cc$ a morphism $\hs{-1}\xymatrix{f:P \ar[r] & Q}\hs{-1}$ in $\cc$ is a homomorphism in $R$-$\mb{Mod}$ satisfying the condition below:
	\begin{align}
	&\forall K \leq P, \hs{1} f(K) \ll Q \Leftrightarrow K \ll P.\tag{C1}\label{mmeqmyy1}
	\end{align}
	Every isomorphism in $\mathbb{QP}$ satisfies in the condition (C1). Note that for each morphism $\hs{-1}\xymatrix{f:P \ar[r] & Q}\hs{-1}$ in $\cc$
	
	(i) if $\text{Im}f \ll Q$, then $P \ll P$ and hence $P = \{0\}$;
	
	(ii) $K_f \ll P$.
	
	Let $\hs{-1}\xymatrix{p: P \ar[r] & Q}\hs{-1}$ be a non-zero morphism in $\cc$ satisfying:
	\begin{subequations}
		\begin{align}
		&\forall A \leq P, K_p \lneqq A \Rightarrow K_p \ll A.\tag{C2}\label{mmeqmyy1}	
		\end{align}
	\end{subequations}
Thus for each morphism $\hs{-1}\xymatrix{g:L \ar[r] & P}$, $p g = 0$ implies that $L = \{0\}$. We show that $p$ is quasi mono. To this end let the diagram,
	\begin{align}
\xymatrix{L \ar@{->}^{g} @<3pt> [rr]
	\ar@{->}_{h} @<-3pt> [rr] & & P \ar[rr]^{p} & & Q} \label{meqmyy3}
	\end{align}
	in $\cc$ such that $p g = p h$ and $L \neq 0$ be given . So we have:
		\begin{align}\label{r3}
& K_p + \text{Im}g = K_p + \text{Im}h.
	\end{align}
Since $\text{Im}g \nleq K_p$ and $\text{Im}h \nleq K_p$, $K_p \lvertneqq K_p + \text{Im}g$ and hence by (C2) $K_p \ll K_p + \text{Im}g$. Thus equality (4) implies that $\text{Im}h = K_p + \text{Im}g$ and so $\text{Im}g \leq \text{Im}h$. Similarly $\text{Im}h \leq \text{Im}g$. Therefore $\text{Im}g = \text{Im}h$. Since $L$ is quasi-projective, there exist morphisms $\hs{-1}\xymatrix{\alpha , \beta : L \ar[r] & L}\hs{-1}$ such that the following diagram commutes
	$$\xymatrix{L \ar@ {-{>>}}[dr]_{\bar{g}} \ar@<1ex> @{.>} [rr]^{\hs{1} \alpha} & & L \ar@ {-{>>}}[dl]^{\bar{h}} \ar@<1ex> @{.>} [ll]^{\hs{1}\beta}\\
	& \text{Im}g = \text{Im}h & }$$
where for each $x \in L$, $\bar{g}(x) = g(x)$ and $\bar{h}(x) = h(x)$. It is easy to see that $\alpha$ and $\beta$ are morphisms in $\cc$ and $\la g \ra = \la h \ra$. Therefore $p$ is quasi mono in $\cc$.
\end{example}

 Recall that a semiperfect ring is one for which every finitely generated module has a projective cover, (see \cite[p.179]{r}). Semiperfect rings $T$, whose indecomposable, projective left and right
modules have simple and essential socles, are called $QF$-2 rings, (see \cite[p.557]{w}).

 \begin{example}\label{QF-2}
Let $R$ be a $QF$-2 ring and $P$ be a projective cover of a simple module. Since every simple module is trivially quasi-projective and indecomposable, $P$ is indecomposable, (see \cite[Exercises 17, p.203]{af}). Thus $P$ has a simple and essential socle and hence each non-zero submodule of $P$ contains $\text{Soc}(P)$. Furthermore $\text{Soc}(P)$ is a fully invariant submodule of $P$, (\cite[p.16]{t}), and so $P/\text{Soc}(P)$ is quasi-projective, (see \cite[Lemma 4.2]{rv}). Now consider the canonical projection $\xymatrix{p: P \ar@{-{>>}} [r] & P/\text{Soc}(P)}$. Let $K \leq P$, $p(K) \ll P/\text{Soc}(P)$ and $K + L = P$ for some $L \leq P$. Thus $p(K) + p(L) = P/\text{Soc}(P)$ and hence $p(L) = P/\text{Soc}(P)$. Therefore $p^{-1}p(L) = P$. Let $x \in p^{-1}p(L)$, so there exists $l \in L$ such that $p(x) = p(l)$. Thus $x - l \in K_p = \text{Soc}(P)$. Since $\text{Soc}(P) \leq L$, $x - l \in L$. This implies that $x \in L$ and so $p^{-1}p(L) = L$. Therefore $L = P$ and hence $K \ll P$. Since $p$ is an epimorphism and $K_p$ is simple and essential, it fulfills the conditions (C1) and (C2) of \Cref{ex.m.1} and hence $p$ is quasi mono in $\cc$.
\end{example}

 \begin{example}\label{yosif}
(1) (see \cite[p.36]{c} and \cite[Example 2.18]{yy}). Let $R = \mathbb{Z}_2[x_1, x_2, \cdots]$ where $x_i^3 = 0$ for all $i$, $x_i x_j = 0$ for all $i \neq j$ and $x_i^2 = x_j^2 = m \neq 0$ for all $i$ and $j$. Then $R$ is a commutative local ring and $R$ has a simple essential socle
$J(R)^2 = \mathbb{Z}_2 m$ as $R$-module. In particular, $R$ is uniform. Note that $\text{Soc}(R) \subseteq J(R)$. Thus $R$ has no non-zero semisimple direct summand, hence every simple $R$-submodule of $R$ is superfluous, (see \cite[Lemma 2.4]{ms}). Therefore $\text{Soc}(R) \ll R$. As \Cref{QF-2} the canonical projection$\xymatrix{p: R \ar@{-{>>}} [r] & R/\text{Soc}(R)}$ fulfills the conditions (C1) and (C2) of \Cref{ex.m.1} and hence $p$ is quasi mono in $\cc$.

 (2) In (1) if we replace $\mathbb{Z}_2$ by an ordered field, the results are still valid, (see \cite[Example 20]{ab}).

 (3) Since for all $n \in \mathbb{Z}$, $n\mathbb{Z}$ is a left $\mathbb{Z}$- right $\text{End}(_{\mathbb{Z}}\mathbb{Z})$ submodule of $_{\mathbb{Z}}\mathbb{Z}$, $\mathbb{Z}/n\mathbb{Z}$ is a quasi-projective $\mathbb{Z}$-module (see \cite[Exercises 16(3), p.203]{af}). Therefore $\mathbb{Z}_8$ is a quasi-projective $\mathbb{Z}$-module and $\text{Soc}(\mathbb{Z}_8) = \{\bar{0}, \bar{4}\}$. Thus $\text{Soc}(\mathbb{Z}_8)$ simple and essential $\mathbb{Z}$-submodule of $\mathbb{Z}_8$. Also, $\mathbb{Z}_8/\text{Soc}(\mathbb{Z}_8) \cong_{\mathbb{Z}} \mathbb{Z}_4$ and hence $\mathbb{Z}_8/\text{Soc}(\mathbb{Z}_8)$ is a quasi-projective $\mathbb{Z}$-module. As in \Cref{QF-2}, the canonical projection $\xymatrix{\mathbb{Z}_8 \ar@{-{>>}} [r]^{\hs{-8}p} & \mathbb{Z}_8/\text{Soc}(\mathbb{Z}_8)}$ fulfills the conditions (C1) and (C2) of \Cref{ex.m.1} and hence $p$ is quasi mono in $\cc$. Note that the lattice of submodules of $\mathbb{Z}_8$ is a completely ordered set by inclusion. This gives other examples of quasi mono.
\end{example}

\begin{example}\label{e:quasi monomo}
	Let $X$ and $X_0$ be two sets with $X_0 \varsubsetneq X$ and $\cc$ be a subcategory of $\mb{Set}$, such that $X_0 \in \cc$ and $X \notin \cc$.
	Now define the subcategory $\cd$ of $\mb{Set}$ to have $obj (\cc) \cup \{X\}$ as objects and for all $ A,B \in \cd$, $ Hom_{\cd}(A,B) =$
	
	$$\Small{ \begin{cases}
	Hom_{\cc}(A,B) & A,B \in \cc \cr Hom_{Set}(A,B) & A=X , B \in \cc \cr \emptyset & A \in \cc , B =X \cr \{ \xymatrix{X \ar[r]^{f} & X} \mid f(X_0)\su X_0, \xymatrix{X-X_0\hs1
		\ar@{>-{>>}}[rr]^{f\upharpoonright(X-X_0)} & & X-X_0} \} & A=X , B=X
	\end{cases}} $$
	where $\hs{-1}\xymatrix{X-X_0\hs1\ar@{>-{>>}}[rr]^{f\upharpoonright(X-X_0)} & & X-X_0}\hs{-1}$ is the restriction of $f$ to $X-X_0$ and is assumed to be a bijection to $X-X_0$. Let $\hs{-1}\xymatrix{X_0 \ar@{^{(}->}[r]^{j} & X}\hs{-1}$ be an inclusion map and so there exists $\hs{-1}\xymatrix{X \ar@{-{>>}}[r]^{\hs{-2}r} & X_0}\hs{-1}$ such that $r j = 1_{X_0}$. Then $r$ is a quasi mono in $\cd$. To prove this suppose the diagram,
	$$ \xymatrix{X \ar@{->}^{f} @<3pt> [rr]
		\ar@{->}_{g} @<-3pt> [rr] & & X \ar[r]^{r} & X_0}$$
	is given in $\cd$ such that $r f = r g$. Define $\hs{-1}\xymatrix{h: X \ar[r] & X}\hs{-1}$ by
	$$h(x)
	:=\begin{cases}
	x & x\in X_0 \\
	t & x\in X-X_0 \text{ and } g\upharpoonright(X-X_0)(t) = f\upharpoonright(X-X_0)(x).
	\end{cases}$$
	It is easy to see that $f = g h$. Thus $f \in \la g \ra$. Similarly $g \in \la f \ra$ and hence $\la f \ra = \la g \ra$.
\end{example}

 In the following examples we give a quasi right $\cm$-factorization structure in which each $m$ in $\cm$ is a quasi mono.

\begin{example}\label{qpart}
Consider the category $R$-$\mb{Mod}$, where $R$ is a semisimple ring. Now define the
class $\cm$ as follows:
$$\cm = \{m \mid m \hbox{ is mono or superfluous epi}\}.$$
Then $R$-$\mb{Mod}$ has quasi right $\cm$-factorizations.

Now let $M$ be a supplemented module and $0 \neq K \ll M$. Let $\xymatrix{p: M \ar[r] & M/K}$be the canonical projection. Consider the following commutative diagram
$$\xymatrix{R^{(\Lambda)} \ar@{.{>>}} [dr]_{f} \ar@{-{>>}}[r]^{\hs{-3}q} & M/K\\
	& M \ar@{-{>>}} [u]_p}
$$
where $M/K$ is the homomorphic image of a free $R$-module $R^{(\Lambda)}$ for some set $\Lambda$, $p$ is the canonical projection and $q = p f$. Since $pf$ is an epi and $p$ is a superfluous epi, then $f$ is epi, (see \cite[Corollary 5.15.]{af}). As we have shown in \Cref{supp}, $\xymatrix{M \ar@{-{>>}}[r]^{\hs{-3}p} & M/K}$ is a quasi mono. Note that since $R$ is semisimple, $M$ is hereditary.	Now consider the following unbroken commutative diagram in $R$-$\mb{Mod}$, where $n \in \cm $.
$$\xymatrix{R^{(\Lambda)} \ar[r]^{u} \ar@{-{>>}}[d]_{f} & L \ar[d]^{n}\\
M \ar@{.>}[ur]^{d} \ar@{-{>>}}[r]_{\hs{-2}p} \ar @{} [ur] |{ } & M/K} $$
where $p f = nu$ and so $n$ is epi. Since $M$ is projective, there exists a morphism $d$ such that $n d = p$. Thus the factorization $q = p f$ is a quasi right $\cm$-factorization and $p$ is a quasi mono.
\end{example}

 \begin{example}
Recall that a ring $R$ is completely hereditary if its class of quasi-projective
modules is closed under taking submodules, (see \cite{fu}). Let $\cc$ be as in Example~\ref{ex.m.1} and
$$\cm = \{m \in \cc \mid m \hbox{ is mono or split epi in $\cc$}\}.$$
Let$\xymatrix{f: P \ar[r] & Q}$ be a morphism in $\mathbb{QP}$. Since $P/K_f$ is isomorphic to a submodule of $Q$, $P/K_f$ is quasi-projective. Let$\xymatrix{p: P \ar[r] & Q}$ be a morphism in $\cc$. Consider the factorization
$$\xymatrix{P \ar[r]^{\hs{-5}p} & Q = P \ar@{-{>>}} [r]^{\bar{p}} & P/K_p\hs{1} \ar@{{>}->}[r]^{\hs{3}q} & Q}$$
in $R$-$\mb{Mod}$, where $\bar{p}$ is the canonical projection and for each $x \in P$, $q(x+K_p) = p(x)$. By \cite[Proposition 5.17.]{af}, the morphisms $\bar{p}$ and $q$ are in $\cc$ and satisfy the condition (C2) of the \Cref{ex.m.1}. Thus the factorization $p = q \bar{p}$ is a quasi right $\cm$-factorization of $p$. 

Now let $R$ also be left artinian ring and $R/J$ be projective $R$-module, where $J = J(R)$ is the Jacobson radical of $R$. Thus there is a complete set of pairwise orthogonal primitive idempotents $e_1, \dots ,e_n$ such that $R = Re_1 \oplus \cdots \oplus Re_n$ and each $Re_i/Je_i$ is simple and$\xymatrix{Re_i \ar[r]^{\hs{-5}p_i} & Re_i/Je_i}$ is a projective cover, (see \cite[Exercises 20, p.203]{af}). Therefore for each $i$, $Je_i$ is a superfluous maximal submodule of $Re_i$. Thus we have the following commutative diagram.
\begin{equation*}
\xymatrix{R \ar@{.{>}} [dr]_{f_i} \ar@{-{>>}}[r]^{\hs{-7}q} & Re_i/Je_i\\
	& Re_i \ar@{-{>>}} [u]_{p_i}}
\end{equation*}
where for each $r \in R$, $q(r) = re_i + Je_i$. By \cite[Exercises 21.17(3), p.183]{w}, for each $i$, $Je_i = \text{Rad}(Re_i)$ and by \cite[21.6(2)]{w}, it is a fully invariant submodule of $Re_i$. So $Re_i/Je_i$ is a quasi-projective module. Since $q$ is epi and $p_i$ is a superfluous epi, $f_i$ is epi, (see \cite[Corollary 5.15.]{af}). Note that \cite[Proposition 5.17.]{af} implies that $q$ is a morphism in $\cc$. Now we show that for each $i$, $f_i$ is a morphism in $\cc$. Let $N \ll R$. Since $f_i$ is an epimorphism, $f_i(N) \ll Re_i$. Now, if $f_i(N) \ll Re_i$, then $q(N) = p_i(f_i(N)) \ll Re_i/Je_i$ and so $N \ll R$. So the above diagram is in $\cc$. Also for each $i$, since $Je_i$ is superfluous and maximal submodule of $Re_i$, $p_i$ satisfies the condition (C2) and hence is a quasi mono. Since $R/J \cong Re_1/Je_1 \oplus \cdots Re_n/Je_n$ and $R/J$ is a projective $R$-module, for each $i$, $Re_i/Je_i$ is a projective $R$-module and hence $p_i$ is split epi. Thus there exists$\xymatrix{s_i: Re_i /Je_i \ar[r] & Re_i}$such that $p_i s_i = 1_{Re_i/Je_i}$. Now we show that $s_i$ is a morphism in $\cc$. Let $L/Je_i \ll Re_i/Je_i$ and $s_i(L/Je_i) + K = Re_i$. Thus $L/Je_i + p_i(K) = Re_i/Je_i$ and so $(K + Je_i)/Je_i = p_i(K) = Re_i/Je_i$. Therefore $K + Je_i = Re_i$ and hence $K = Re_i$. This implies that $s_i(L/Je_i) \ll Re_i$ and hence $s_i$ is a morphism in $\cc$. Therefore for each $i$, $p_i\in \cm$. Similar to \Cref{qpart}, the factorization $q = p_i f_i$ is a quasi right $\cm$-factorization.
\end{example}

Let $\cc$ be a class of morphisms in a category $\cx$ which is closed under composition. We write $f \leq_{\cc} g$, whenever there exists an $h \in \cc$ such that $f = gh$; and we write  $f \sim_{\cc} g$, whenever $f \leq_{\cc} g$ and $g \leq_{\cc} f$.

\begin{example} \label{e:construct quasi M}
	Suppose $\cc$ is a category and $\cm$ is a class of quasi monos in $\cc$ such that it is a quasi right factorization in $\cc$. Let $\cx$ be a category such that $ \cc \subseteq \cx$ and there exists $ X \in obj(\cx) - obj(\cc)$ and for all $A \in \cc$ and $\hs{-1}\xymatrix{g,h: X \ar[r] & A }\hs{-1}$ there exists an isomorphism $\hs{-1}\xymatrix{k: X \ar[r] &X}\hs{-1}$ in $\cx$ such that $ g = hk$.
	Now define the subcategory $\cd$ of the category $\cx$  to have $ obj(\cc) \cup \{X\}$ as objects and for all $ A, B \in \cd$,
	$${\tiny \text{Hom}_{\cd }(A,B) = \begin{cases}
		\text{Hom}_{\cc }(A,B) & A, B  \in \cc  \cr  \text{Hom}_{\cx}(A,B) &  A = X, B\in \cc \cr \emptyset & A  \in \cc , B= X \cr  \{\hs{-1} \xymatrix{X \ar[r]^f & X}\hs{-2}\mid f \hbox{ is an isomorphism in } \cx \} & A=X , B= X
		\end{cases}} $$
	First we show that the class $ \cn = \cm \cup \{\hs{-1} \xymatrix{X \ar[r]^f & A }\hs{-2} : A \in \cd \}$ is a class of quasi monos in $\cd$.
	Suppose that $\hs{-1} \xymatrix{B \ar[r]^f & C }\hs{-1}$ in $ \cn$ and $\hs{-1}\xymatrix{g,h: A \ar[r] & B }\hs{-1}$ in $ \cd$ are given such that $ fg = fh $. 	We have two cases:
	
	$(i)$ Suppose $ f \in \cm$. If $ A \in \cc$, then  $g \sim_{\cc} h$. So there exist morphisms $\hs{-1}\xymatrix{a,a': A \ar[r] & A }\hs{-1}$ in $\cc$ such that $ g = ha$ and $ h = g a'$. Since $a,a'$ are in $\cd$,  $g \sim_{\cd} h$. If $ A=X$, then there exists an isomorphism $\hs{-1} \xymatrix{k:X \ar[r]  & X}\hs{-1}$ in $\cx$ such that $ g = hk$. It follows that $k \in \cd$, thus $g \sim_{\cd} h$.
	
	$(ii)$ If $ B=X$, then $A=X$ and $g$ and $h$ are isomorphisms. Hence, $g \sim_{\cd} h$.
	
	Next we show that $\cd$ has quasi right $\cn$-structure. To this end,  let  $\hs{-1} \xymatrix{f : A \ar[r]  & B}\hs{-1} $ be a morphism in $ \cd$.	We have two cases:
	
	$(i)$ If $ A \in \cc$, then $f$ is a morphism in $\cc$ and since $\cc$ has quasi right $\cm$-factorizations, there exist morphisms $\hs{-1}\xymatrix{m_f : M \ar[r] & B} \hs{-1} $ and $\hs{-1}\xymatrix{e_f : A \ar[r] & M}\hs{-1}  $  in $ \cc$ such that $ f = m_f e_f$ and $ m_f \in \cm$. Since $\cm \subset \cn$, $ f = m_f e_f$ is a factorization in $\cd$ as well.
	
	$(ii)$ If $ A=X$, then set $ m=f$ and $ e =1_A$ so that $ f= me$ is a factorization for $f$ in $\cd$.
	
	Now consider the following unbroken commutative diagram in $\cd$ such that $n \in \cn$.
	
	$$\xymatrix{A \ar[r]^{u} \ar[d]_{e } &  N \ar[d]^{n}\\
		M \ar@{.>}[ur]^{d} \ar[r]_{\hs{2} m }\ar @{} [ur] |{\hs{8}///} & B} $$
	We have two cases:
	
	$(i)$ If $A \in \cc$, then the above diagram commutes in $\cc$ and $ e = e_f , m=m_f$ and $n \in \cm$. So there exists a morphism  $\hs{-1}\xymatrix{d : M \ar[r] & N}\hs{-1}  $ in $\cc$ such that $m=nd$. Hence $d$ is a morphism in $\cd$ and $m \leq_{\cd} n$.
	
	$(ii)$ If $A=X$, then $e=1_A$ and $ m=f$. Set $d=u$, so $m=nu$ and  $m \leq_{\cd} n$.
	
	Hence, $\cd$ has quasi right $\cn$- factorization structure.
\end{example}

In the following two examples we give a class of quasi epis which are not epis.

\begin{example}
	Let $\cc$ be the category of subrings of real numbers $\mathbb{R}$. Then the inclusion $ \xymatrix{\mathbb{Z} \ar@{^{(}->}[r]^{\hs{-4}j} & \mathbb{Q}(\sqrt{p})}$, where $p$ is a prime number, is  quasi epi. To show this, suppose the diagram,
	$$ \xymatrix{\mathbb{Z} \ar@{^{(}->}[rr]^{\hs{-5}j} & &
		\mathbb{Q}(\sqrt{p}) \ar@{->}^{f} @<3pt> [rr]
		\ar@{->}_{g} @<-3pt> [rr] & & S}$$
	is given in $\cc$ such that $f j = g j$. Thus for each $m \in \mathbb{Z}$, $f(m) = g(m)$ and thus for each $a \in \mathbb{Q}$, $f(a) = g(a)$. On the other hand $f(\sqrt{p}) = g(\sqrt{p})$ or $f(\sqrt{p}) = - g(\sqrt{p})$. In the first case $f = g$. In the second case define $\hs{-1}\xymatrix{h:\mathbb{Q}(\sqrt{p}) \ar[r] & \mathbb{Q}(\sqrt{p})}\hs{-1}$ by $ h(a + b \sqrt{p}) = a - b\sqrt{p}$, implying $f = g h$ and $g = f h$.
\end{example}

 \begin{example}\label{e:quasi epimo}
	Let $X$ and $X_0$ be in $\mb{Set}$ with $X_0 \varsubsetneq X$ and $\cc$ be a subcategory of $\mb{Set}$, such that $X-X_0 \in \cc$ and $X \notin \cc$. Now define the subcategory $\cd$ of $\mb{Set}$ to have $obj (\cc) \cup \{X\}$ as objects and $ Hom_{\cd}(A,B) =$
	
	$$\begin{cases}
	Hom_{\cc}(A,B) & A,B \in \cc \cr Hom_{Set}(A,B) & A \in \cc , B=X \cr
	\emptyset & A=X, B\in \cc \cr
	\{ \xymatrix{X \ar[r]^f & X} \mid \hs{-1}\xymatrix{X_0 \hs1 \ar@{>-{>>}}[rr]^{f\upharpoonright{X_0}} & & X_0} \} & A=X, B=X
	\end{cases} $$
	Then the inclusion map $\hs{-1}\xymatrix{X-X_0 \ar@{^{(}->}[r]^{\hs{4}j} & X}\hs{-1}$ is a quasi epi in $\cd$.
\end{example}

 \section{General closure operator}
Consider a category $\cx$ and a fixed class $\cm$ of morphisms (not necessarily monomorphisms) in $\cx$ which we think of as generalized subobjects. For every object $X$ of $\cx$, let $\mc{M} / X$ be the class of all morphisms in $\cm$ with codomain $X$. The relation given by
$$m \leqslant n \Leftrightarrow \la m \ra \su \la n \ra$$
is reflexive and transitive, hence $\mc{M} / X$ is a preordered class. Note that $m \leqslant n$ means that there exists a morphism $j$ such that $m = n j$. Since $\cm$ is an arbitrary class of morphisms, $j$ is not uniquely determined. Also $m \leqslant n$ and $n \leqslant m$ do not imply that $j$ is an isomorphism. If $m \leqslant n$ and $n \leqslant m$, then we say $m$ and $n$ are $\la \hs{1} \ra$-equal and we write $m \sim n$. Of course, $\sim$ is an equivalence relation, and $\mc{M}/X$ modulo $\sim$ is a partially ordered class. In fact, we shall use the notations $\leqslant$ and $\sim$ for elements of $\mc{M}/ X$ rather than for their $\sim$-equivalence classes. So, with $\underline{m}$ denoting the $\sim$-class of $m$, we have
$$m \leqslant n \Leftrightarrow \underline{m} \leqslant \underline{n}$$
$$m \sim n \Leftrightarrow \underline{m} = \underline{n}.$$

 \begin{definition}\label{def.cl} Suppose that $\cx$ has quasi right $\cm$-factorizations. A {\it general closure operator} $\mb C$ on $\mc X$ with respect to $\mc M$ is given by $\mb C = (\mb {c}_{X})_{X\in \mc X}$, where $\mb{c}_{X}:\hs{-2}\xymatrix{ \mc {M}/X \ar[r] & \mc {M}/X }\hs{-1}$ is a map satisfying:
	
	(a) the extension property: for all $m\in \mc{M}/X$, $m\leqslant \mb{c}_{X}(m)$;
	
	(b) the monotonicity property: for all $m, m'$ in $\mc{M}/X$ whenever $m\leqslant m'$ in $\mc{M}/X$, then $\mb{c}_{X}(m)\leqslant \mb{c}_{X}(m')$ ;
	
	(c) the continuity property: for every morphism $\hs{-1}\xymatrix{f: X\ar [r] & Y}\hs{-1}$ in $\mc X$ and for all $m\in \mc{M}/X$, $m_{f\mb{c}_{X}(m)} \leqslant \mb{c}_{Y}(m_{fm})$.
\end{definition}

 Since $\mc{M}$ is an arbitrary class of morphisms and $\cx$ has quasi right $\cm$-factorizations, the Definition~\ref{def.cl} is a generalization of the closure operator that is defined in~\cite{dt}.

 \begin{remark}
	Suppose that $\cm$ has $\cx$-pullbacks (i.e. $m\in \cm$ and $f\in\cx $ with the same codomains, implies the pullback, $f^{-1}(m)$, of $m$ along
	$f$ is in $\cm$) and $\cx$ has quasi right $\cm$-factorizations. For every object $X$ in $\cx$ consider the preordered class $\mc{M} / X$. We have the adjunction,
	\begin{equation}\label{eq:adjunction}
	\xymatrix{\mc{M}/X \ar@<1ex> [rr]^{\hs{4} f(-)}_{\hs{4}\perp} & & \mc{M}/X \ar@<1ex> [ll]^{\hs{4}f^{-1}(-)}
	}
	\end{equation}
	see~\cite{mh11}. In the presence of (b), by adjunction~(\ref{eq:adjunction}), the continuity condition can equivalently be expressed as:
	
	(c$'$) for every morphisms $\hs{-1}\xymatrix{f: X\ar [r] & Y}\hs{-1}$ in $\mc X$ and for all $m\in \mc{M}/Y$, $$\mb{c}_{X}(f^{-1}(m))\leqslant f^{-1}(\mb{c}_{Y}(m)).$$
\end{remark}

 \begin{example}\label{examp}
	(1)
	Let $\cx$ be a pointed category with finite products, and $\ck$ be a
	non-empty class of objects of $\cx$ such that for any pair of isomorphic
	objects either both are in $\ck$ or both are not; and let $\cm$ be the
	class of all spilt epis with kernel in $\ck$. Then $\cm$ is a
	quasi right factorization structure in $\cx$ which is closed under
	composites with isomorphisms on both sides. A morphism $\hs{-1}\xymatrix{f:
	X \ar[r] & Y}\hs{-1}$ in $\cx $ can be factored as $f = \pi_2 \la 0,f \ra$,
	where$\xymatrix{\la 0,f \ra: X \ar[r] & K \times Y}$for some $K \in \ck$.
	It is easy to see that $ \pi_{2}$ is a spilt epi with kernel in $\ck$.
	Since $\cm$ is a collection of split epis, any family $\mb{C} = (\mb{c}_{X})_{X
	 \in \cx}$ where $ \mb{c}_{X}$ is a map from $\mc{M}/X$ to itself, forms a
	 general closure operator on $\cx$ with respect to $\cm$.
	
	Examples (2), (3) and (4) below are special cases of this type.

 	(2) Consider the category $_{R}M_{S}$ of $(R,S)$-bimodules, where $R$ and $S$ are commutative rings and suppose that there exists a ring homomorphism $\hs{-1}\xymatrix{\sigma: R \ar[r] & S}\hs{-1}$ such that $\sigma(1_R) = 1_S$. Thus $S$ is an $R$-module by $r \cdot s = \sigma(r) s$ and hence $S \in{ _{R}}M_{S}$. Suppose that $\cc$ is a full subcategory of $_{R}M_{S}$ whose objects are $(R,S)$-bimodules $M$ such that for each $r\in R$, $s\in S$ and $m \in M$ we have $s(rm) = (s \cdot r)m$. Let $\cm $ be the class of all split epis in $\cc$. One can easily verify that $\cm$ is a quasi right factorization structure in $\cc$. A morphism $\hs{-1}\xymatrix{f: X \ar[r] & Y}\hs{-1}$ in $\cc$ can be factored as follows:
	$$\xymatrix{ X \ar[dr]_{\la 0,f \ra} \ar@{} [drr] |{///}
			\ar[rr]^{f} & & Y \\
			& X \oplus Y \ar[ur]_{\pi_{2}} & }$$
	For each morphism $\hs{-1}\xymatrix{\varphi : M \ar@{>>} [r] & X}\hs{-1}$ in $\cm$, define its closure $ \bar{\varphi}$ be the unique map making the following diagram commute,
	$$\xymatrix{S \times M \ar[rr]^{\otimes_R} \ar[dr]_\psi \ar @{} [drr] |{\hs{-4}///} & & S \otimes_R M \ar[dl]^{\bar{\varphi}}\\
			& X &}$$
	where the map $\hs{-1}\xymatrix{\psi: S\times M \ar[r] & X}\hs{-1}$ takes $ (s,m)$ to $ s \varphi(m)$.
	
	(3) Let $\cc$ be the {\it category of torsion free modules},
	\cite{ej}, and $\cm$ be the class of all split epis. Then $\cm$ is a
	quasi right factorization structure in $\cc$, \cite{mh11}. Suppose that
	$\hs{-1}\xymatrix{m: M \ar@{>>} [r] & X}\hs{-1}$ is a morphism in $\cm$.
	There is a torsion free precover $\hs{-1}\xymatrix{\varphi : T \ar[r] &
	X}\hs{-1}$. Since $M$ is torsion free, there is a map
	$\hs{-1}\xymatrix{\psi : M \ar[r] & T}\hs{-1}$ such that $\varphi \psi =
	m$.
	Now define the closure of $m$ to be the map $\varphi$.

 (4) Let $ \cc $ be {\it an abelian category with enough
injectives}, \cite{f}. The collection $\cm$ of all epis whose kernels are
injective is a quasi right factorization structure. A morphism
$\hs{-1}\xymatrix{f: X \ar[r] & Y}\hs{-1}$ can be factored as $\pi_2 \la i,f
\ra$, where $\hs{-1}\xymatrix{i: X \ar[r] & E}\hs{-1}$ is the mono from $X$ to
an injective object $E$ and $\hs{-1}\xymatrix{\pi_2: E\times Y \ar[r] &
Y}\hs{-1}$ is the projection to the second factor. Now for each morphism
$\hs{-1}\xymatrix{m: M \ar@{>>} [r] & X}\hs{-1}$ define its closure to be the
map $ m\pi_{2} : K\oplus M \twoheadrightarrow X$, where $ K = \text{Ker}(m)$.

(5) Let $ \cc $ be a {\it closed model category}, \cite{q}. The
collection $\cm$ of fibrations in $\c$ is a quasi right factorization structure. For
each object $ X \in \cc $ we have a trivial fibration $\hs{-1}\xymatrix{p_{X}:
Q(X) \ar[r] & X }\hs{-1}$ with $ Q(X)$ cofibrant.
	Now for each morphism $\hs{-1}\xymatrix{m: M \ar[r] & X}\hs{-1} $ in $\cm$ define its closure to be the map $\hs{-1}\xymatrix{m p_{m}: Q(M) \ar[r] & X}\hs{-1}$.
	
(6) As a special case of (5), in the category
{\it $\mb {Top}$}, of topological spaces and continuous maps, the collection
$\cm$ of Serre fibrations is a quasi right factorization structure. Now the
closure of a morphism $\hs{-1}\xymatrix{m: M \ar[r] & X} \hs{-1}$ in $\cm$ is
as in (5).

 (7) Let $ \cc $ be a {\it model category}. For the {\it category
of fibrant objects}, $\cc_{f} $, the collection $\cm$ of fibrations is a
quasi right factorization structure. Define the closure of $\hs{-1}\xymatrix{m:
M \ar[r] & X}\hs{-1}$ in $\cm$ to be the projection to the first factor,
	$\hs{-1}\xymatrix{\pi_{1}: X \times M \ar[r] & X}\hs{-1}$.

 (8) As a special case of (7), in the category
{\it $\mb {Top}$}, in which all the objects are fibrant, the collection
$\cm$ of Serre fibrations is a quasi right factorization structure. Define the
closure of $\hs{-1}\xymatrix{m: M \ar[r] & X} \hs{-1}$ in $\cm$ to be the
projection to the first factor,
	$\hs{-1}\xymatrix{\pi_{1}: X \times M \ar[r] & X}\hs{-1} $.

 (9) \label{exm.8}In the {\it cofibrant category ($\mb{Top}$, cofibrations,
homotopy equivalences)}, the collection $\cm$ of homotopy equivalences is a
quasi right factorization structure. A morphism $f$ can be factored
as$$\xymatrix{X \ar[r]^{\hs{-4} f} & Y = X \ar[r]^{\hs{4}i_f} & Z_f \ar[r]^{r_f}
& Y,}$$where $r_f$ is a homotopy equivalence, $i_f$ is a cofibration and $Z_f$
is the mapping cylinder of $f$, \cite{ksvw}. For each morphism
$\hs{-1}\xymatrix{M
\ar[r]^{m} & X}\hs{-1} $ in $\cm$ define its closure to be the map,
$\hs{-1}\xymatrix{Z_m \ar[r]^{r_m} & X}\hs{-1} $.

 (10) In the {\it fibrant category ($\mb{Top}$, fibrations,
homotopy equivalences)}, the collection $\cm$ of fibrations is a quasi right
factorization structure. A morphism $f$ can be factored
as$\xymatrix{X \ar[r]^{\hs{-4} f} & Y = X \ar[r]^{\hs{4}q_f} & P_f \ar[r]^{k_f}
	& Y,}$where $q_f$ is a homotopy equivalence, $k_f$ is a fibration and $P_f$
	is the mapping path space of $f$, \cite{ksvw}. For each morphism
	$\hs{-1}\xymatrix{m: M \ar[r] &
X}\hs{-1} $ in $\cm$ define its closure to be the map, $\hs{-1}\xymatrix{k_m:
P_m \ar[r] & X}\hs{-1}$.

 (11) In the {\it Kleisli category
$\mb{Set_\mathbb{P}}$}, where $P$ is the power set monad
$\mathbb{P}=(P,\eta,\mu)$,
	for each morphism $\hs{-1}\xymatrix{\hat{f}: X \ar[r] & Y}\hs{-1}$ in
	$Set_\mathbb{P}$, let $\hs{-1}\xymatrix{f: X \ar[r] & P(Y)}\hs{-1}$ be its
	associated morphism in $\mb{Set}$ and$$\xymatrix{X \ar[r]^{\hs{-6}f} & P(Y) = X
	\ar[r]^{\hs{7}f'} & I_{f} \ar[r]^{\hs{-3}m_{f}} & P(Y)}$$be the $(Epi,
	Mono)$ factorization of $f$.
	The class $\cm=\{\widehat{m_f}: \hat{f}\in Set_\mathbb{P}\}$ is a quasi
	right factorization structure, see \cite{mh11}. For each morphism
	$\hs{-1}\xymatrix{\widehat{m_f}: I_{f} \ar[r] & Y}\hs{-1} $ in $\cm$ define
	its closure to be the map $$\xymatrix{\widehat{\mu_{Y} P(m_{f})} : P(I_{f})
	\ar[r] & Y}\hs{-1}. $$

 (12) Let $\cx$ be a category with binary coproducts. Then the
class $\cm=\{\ [f,f] : f \hbox{ is a morphism in } \cx \}$ is a quasi right
factorization structure. A morphism$\xymatrix{X \ar[r]^f & Y}\hs{-1}$
can be factored as $f=[f,f]\nu_1$, where $\nu_1$ is the injection of the
coproduct. For each morphism $ m$ in $\cm$ define its closure to
be the map $[m,m]$.

 (13)
	Given any adjunction
	$\xymatrix{\cx \ar@<1ex> [rr]^{\hs{1} F}_{\hs{1}\perp} & & \ca \ar@<1ex> [ll]^{\hs{1}G}
	}$.	Let $\tilde{ \ca} $ be the full subcategory of $ \ca$ consisting of
	those objects $A$ such that$\xymatrix{\epsilon_{A} : F G(A)\ar[r] & A}$ the
	component of the counit (of the above adjunction) at $A$, is a split epi.
	The class $\cm$ consisting of those split epis in $\ca$ whose
	domain is $F(X)$ for some object $X$ in $\cx$ is a quasi right
	factorization structure on $\tilde{ \ca} $. Each morphism $f$ in $\tilde{
	\ca} $ factorizes as$$\xymatrix{A \ar[r]^{\hs{-4} f} & B = A
	\ar[r]^{\hs{-2}sf} & FG(B) \ar[r]^{\hs{4}\epsilon_B}& B,}$$where $s$ is any
	splitting of $\epsilon_B$. If in addition $\cx$ has binary
	coproducts, for each morphism $\hs{-1}\xymatrix{m: F(X) \ar[r] & A}\hs{-1}$
	in $\cm$ define its closure to be the unique morphism
	$\hs{-1}\xymatrix{c_{X}(m): F(X + X) \ar[r] & A}\hs{-1}$, corresponding by
	adjunction to $\hs{-1}\xymatrix{X + X \ar[r]^{[\bar{m},\bar{m}]} &
	G(A)}$, where $\hs{-1}\xymatrix{X \ar[r]^{\hs{-4}\bar{m}} & G(A)}\hs{-1}$ is
	the morphism corresponding by adjunction to $m$. Now we show that $c_{X}(m)$ is a split epi. Let
	$$\xymatrix{\text{Hom}_{\ca}(F-, -) \ar[r]^{\hs{-10}\theta_{-,-}} & \text{Hom}_{\cx}(- , G-): \cx^{\text{op}} \times \ca \ar[r] & \text{\bf{Set}}}$$
be the natural isomorphism corresponding to the adjunction $F \dashv G$. For $\hs{-1}\xymatrix{X \ar[r]^{\hs{-3}\bar{m}} & G(A)}\hs{-1}$ we have

 $\tiny{\xymatrix{
		X \ar [d]_{\bar{m}} \\G(A), }}$
$\tiny{\xymatrix{\text{Hom}_{\ca}(FG(A), A) \ar @{}[dr] |{_{///}} \ar[d] \ar[r]^{\theta_{GA,A}}
		& \text{Hom}_{\cx}(G(A), G(A)) \ar[d] & \\
		\text{Hom}_{\ca}(F(X), A) \ar[r]^{\theta_{X,A}} & \text{Hom}_{\cx}(X, G(A))
		\ar @{}[ur]| {\quad\quad\quad\quad_{\longrightarrow}}
}}$
$\tiny{\xymatrix{ \epsilon_{A} \ar@{|-{>}}[d] \ar@{|-{>}}[r]
		& 1_{G(A)} \ar@{|-{>}}[d] \\
		\epsilon_{A} F(\bar{m}) \ar@{|-{>}}[r] & \bar{m}
}}$

 \noindent
Since $\theta_{X,A}(m) = \bar{m}$ and $\theta_{X,A}$ is one-to-one, $m = \epsilon_{A} F(\bar{m})$. Suppose that $\xymatrix{\iota_1, \iota_2:X \ar[r] & X+X}$are the canonical injections of the coproduct $X+X$. Since left adjoint preserves coproducts, $F([\bar{m},\bar{m}])$ is the unique morphism such that $F([\bar{m},\bar{m}]) F(\iota_1) = F([\bar{m},\bar{m}]) F(\iota_2) = F(\bar{m})$. Thus $m = \epsilon_{A} F(\bar{m}) = \epsilon_{A} F([\bar{m},\bar{m}]) F(\iota_1) = c_{X}(m) F(\iota_1)$ and hence $c_{X}(m)$ is split epi, because $m$ is split epi.

 (14) As a special case of (13),
consider $\mb {Proj}(R$-$\mb{Mod})$ as the full subcategory of the category
$R$-$\mb{Mod}$, consisting of all projective $R$-modules. The collection $\cm$
of all epis with free domains is a quasi right factorization structure, see
\cite{mh11}. For each
morphism $\hs{-1}\xymatrix{m: F \ar[r] & P}\hs{-1}$ in $\cm$ define its closure
to be the map $\xymatrix{[ m,m ]: F\oplus F \ar[r] & P}$.
	
\end{example}

 Now on instead of saying $\mb C$ is a general closure operator on the category $\mc X$ with respect to $\mc M$ we will say $\mb C$ is a closure operator.

 \begin{lemma} (Weak Diagonalization Lemma)\label{wdl}
	Let $\cx$ has quasi right $\cm$-factorizations and $\mb C$ be a closure operator. For every commutative diagram
	$$\xymatrix{M \ar[r]^{u} \ar[d]_{m} & N \ar[d]^{n}\\
		X \ar[r]_{v} & Y} $$
	with $m, n \in \cm$, there is a morphism $w$ rendering the lower square in the diagram
	$$\xymatrix{ M \ar[d]_{j_m} \ar[rr]^{u} & & N \ar[d]^{j_n} \\
		\mb{c}_X(M) \ar@{} [drr] |{///} \ar[d]_{\mb{c}_X(m)} \ar@ {.>}[rr]^w & & \mb{c}_Y(N) \ar[d]^{\mb{c}_Y(n)} \\
		X \ar[rr]_v & & Y}$$
	commutative.
\end{lemma}

 \begin{proof}
	Let $\xymatrix{\mb{c}_X(M) \ar[rr]^{\hs{-5}v\mb{c}_X(m)} & & Y = \mb{c}_X(M) \ar[r]^{\hs{5}e} & M' \ar[rr]^{m_{v \mb{c}_X(m)}} & & Y,}$ be the $\cm$-quasi right factorization of $v\mb{c}_X(m)$. Since $vm \leq n$, by \Cref{qrf}(b$'$), $m_{vm} \leq n$, hence $m_{v \mb{c}_X(m)} \leq \mb{c}_{Y}(m_{v m}) \leq \mb{c}_Y(n)$. Thus there exists $\xymatrix{w':M' \ar[r] & \mb{c}_Y(N)}$such that $\mb{c}_{Y}(n) w' = m_{v \mb{c}_X(m)}$. Put $w = w' e$, so $\mb{c}_Y(n) w = v \mb{c}_X(m)$.
\end{proof}

 \begin{definition} Suppose that $\mc{M}$ is a class of morphisms in $\cx$ and $\cx$ has quasi right $\cm$-factorizations. Also suppose that $\mb C$ is a closure operator and $m\in \mc{M}\slash X$, where $X$ is an object in $\cx$. We say $m$ is
	
	(a) {\it quasi $\mb C$-closed} in $X$, if $\mb{c}_X(m) \sim m$ (see \cite{mh11});
	
	(b) {\it quasi $\mb C$-dense} in $X$, if $\mb{c}_X(m) \sim 1_X$.
\end{definition}

 A morphism $f$ in $\cx$ is called {\it quasi $\mb C$-dense}, whenever $m_{f}$ is quasi $\mb C$-dense in $\cx$. We denote by $\ce^{ QC}$, the class of all quasi $\mb C$-dense morphisms in $\cx$. Let $\mc{M}^{QC}$ be the class of quasi $\mb C$-closed members of $\mc M$.

\begin{remark}
	Let $f$ and $g$ be two morphisms in the category $\cx$.
	
	(a) $ f \sim 1_X $ is equivalent to $ 1_X \leqslant f $ which is equivalent to $f$ being a split epi.
	
	(b) If $f \leqslant g $ and $f$ is a split epi, then $g$ is a split epi.
	
\end{remark}

 \begin{example}
	(i) In Example~\ref{examp}, (1), (2), (3), (4), (9), (10) and (11)
	members of $\cm$ are all quasi $\mb C$-closed.
	
	(ii) In the Example~\ref{examp}, (1), (2), (3), (4), (7) and (9) members
	of $\cm$ are all quasi $\mb C$-dense.
\end{example}

 \begin{proposition}\label{mqc closed under compo}
	Suppose that $\cm$ has $\cx$-pullbacks and $\mb C$ is a closure operator. Then $\cm^{QC}$ has $\cx$-pullbacks.
\end{proposition}

 \begin{proof}
	 Let $m \in \cm^{QC}$ and the following diagram
	$$\xymatrix{ M^* \ar [d]_{m^*} \ar@{} [dr] |{p.b.} \ar[r]^{f^*} & M \ar [d]^{m} \\
		X \ar[r]_{f} & Y }$$
	be a pullback in $\cx$. Thus $m^* \in \cm$. By \Cref{wdl}, there exists $\xymatrix{ \mb{c}_X(M^*) \ar[r]^{\hs{1}w} & \mb{c}_Y(M)}$such that $\mb{c}_Y(m) w = f \mb{c}_X(m^*)$. Since $\mb{c}_Y(m) \sim m$, there exists$\xymatrix{q: \mb{c}_Y(M) \ar[r] & M}$such that $\mb{c}_Y(m) = m q$. Therefore $f \mb{c}_X(m^*) = \mb{c}_Y(m) w = m (q w)$ and hence there exists a unique morphism$\xymatrix{ \mb{c}_X(M^*) \ar[r]^{\hs{3}r} &M^*}$such that $m^* r = \mb{c}_X(m^*)$. Thus $\mb{c}_X(m^*) \leq m^* \leq \mb{c}_X(m^*)$ and so $m^* \sim \mb{c}_X(m^*)$. Therefore $m^* \in \cm^{QC}$.
\end{proof}

 \begin{proposition}\label{QD is closed under comp. with iso}
	Suppose that $\cx$ has quasi right $\cm$-factorizations and $\mb C$ is a closure operator.
	\begin{itemize}
		\itema Let $\cm$ be closed under composition with isomorphisms on the left. For each morphism $\hs{-1}\xymatrix{X \ar[r]^f & Y}\hs{-1}\in \ce^{ QC}$ and isomorphism $\hs{-1}\xymatrix{Y \ar[r]^{\alpha} & Z}\hs{-1}$ in $\cx$ we have $\alpha f \in \ce^{ QC}$.
		
		\itemb If $f \leqslant g$ and $f \in \ce^{ QC}$, then $g \in \ce^{ QC}$.
	\end{itemize}
\end{proposition}

 \begin{proof}
	(a) By Propositions~\ref{p:composition of qrf by iso} and~\ref{p:composition by iso is not different} and the continuity property of $\mb{C}$ we have $\alpha(\mb{c}_Y (f(1_X))) \leqslant \mb{c}_Z((\alpha f)(1_X))$. Since $\mb{c}_Y (f(1_X)) \sim 1_Y$ and $\alpha(1_Y) \sim 1_Z$, $\alpha f \in \ce^{ QC}$.
	
	(b) Obvious.
\end{proof}

 \begin{remark}
	Suppose that $\cx$ has quasi right $\cm$-factorizations and $\mb C$ is a closure operator.
	
	(a) For each $m, n\in \cm$ if $m \sim n$ and $m$ is quasi $\mb C$-dense, then $n$ is quasi $\mb C$-dense.
	
	(b) If $\mc{M}$ is a class of monos, then $m$ is quasi $\mb C$-closed (dense) if and only if $m$ is $\mb C$-closed (dense).
\end{remark}

 \section{Quasi idempotent and quasi weakly hereditary closure operator}
In this section we define quasi idempotent and quasi weakly hereditary closure operators and we show which one of the examples in the previous section have these properties. Finally we prove under what conditions on the closure operator we have another quasi right(left) factorization structure.

 \begin{definition}
	Suppose that $\cx$ has quasi right $\cm$-factorizations and $\mb C$ is a closure operator. $\mb C$ is called:
	
	(a) {\it quasi idempotent}, if for each $X\in \mc X$ and $m\in \mc{M}/X$, $\mb{c}_{X}(\mb{c}_{X}(m))\sim \mb{c}_{X}(m)$ (see \cite{mh11}).
	
	(b) {\it quasi weakly hereditary}, if for each $X\in \mc X$ and $\hs{-1}\xymatrix{m: M \ar[r] & X}\hs{-1}$ in $ \mc{M}$, there exists a quasi $\mb C$-dense morphism $\hs{-1}\xymatrix{j_m : M \ar[r] & \mb{c}_{X}(M)}\hs{-1}$ such that $\mb{c}_{X}(m) j_m = m$.
\end{definition}

 \begin{example}
	(1) In Example~\ref{examp}, (1), (2), (3), (4), (7), (9), (10) and
	(11) the closure operator is quasi idempotent.
	
	(2) In Example~\ref{examp}, (1), (2), (3), (4), (9), (10) and (11)
	the closure operator is quasi weakly hereditary.
\end{example}

 With $(\ce , \cm)$-factorization structure as defined in \cite{ahs}, we have:

 \begin{theorem} \cite{mh11} Suppose that $\cx$ has $(\ce , \cm)$-factorization structure and $\mb C$ is a quasi idempotent closure operator. Then $\cm^{QC}$ is a quasi right factorization structure for $\cx$.
\end{theorem}

 \begin{theorem}\label{T:closed right}
	Let $\cx$ has quasi right $\cm$-factorizations. For a quasi idempotent closure operator $\mb C$, $\cx$ has quasi right $\cm^{ QC}$-factorizations.
\end{theorem}

 \begin{proof}
	For a given morphism $f$ in $\cx$, $f \leqslant m_f \leqslant \mb{c}(m_f)$ which is quasi closed. That is $\mb{c}(m_f) \in \cm^{QC}$. If $f \leqslant n$ with $n$ quasi closed, then $m_f \leqslant n$. Thus $\mb{c}(m_f) \leqslant \mb{c}(n) \sim n$.
\end{proof}

\begin{remark}
	For a closure operator $\mb C$ we have $ \cm^{ QC} \bigcap \ce^{ QC} = \{f\in \cm \mid f \sim 1\}$. If $\cm\su QM(\cx)$, then $ \cm^{ QC} \bigcap \ce^{ QC} = Iso(\cx) \cap \cm $.
\end{remark}

\begin{proposition}\label{P:left part is dense}
	Suppose $\cm \su QM(\cx)$ is closed under composition. Also suppose $\cx$ has quasi right $\cm$-factorizations and $\mb C$ is a closure operator. If $$\xymatrix{X \ar[r]^{\hs{-4}f} & Y = X \ar[r]^{\hs{3}e} & M \ar[r]^{m_f} & Y}$$ is a quasi right $\cm$-factorization of $f$, then $e \in \ce^{ QC}$ and $e(1_X) , \mb{c}_{M}(e(1_X)) \in Iso(\cx)$.
\end{proposition}

 \begin{proof}
	Since $f \leqslant m_f m_e$, by Lemma \ref{L:quasi right} (a) and (c), $m_f \leqslant m_f m_e$. By Proposition \ref{p:quasi mono}(a) we have $1_M \leqslant m_e$ and so $e \in \ce^{ QC}$. Since $1_M \leqslant m_e \leqslant \mb{c}_{M}(m_e)$, we have $1_M \sim m_e \sim \mb{c}_{M}(m_e)$. Thus $m_e$ and $\mb{c}_{M}(m_e)$ are split epis. Since $m_e$ and $\mb{c}_{M}(m_e)$ are quasi monos, by Lemma~\ref{l:quasi mono split epi is iso} they are isomorphisms.
\end{proof}

 \begin{proposition}\label{P:closure of quasi dense is iso}
	Suppose that $\cx$ has quasi right $\cm$-factorizations and $\cm \su QM(\cx)$. If $\hs{-1}\xymatrix{X \ar[r]^e & M}\hs{-1}$ is quasi $\mb{C}$-dense, then $\mb{c}_{M}(m_e)$ is an isomorphism.
\end{proposition}

 \begin{proof}
	Since $e \in \ce^{ QC}$ and $\mb{c}_{M}(m_e) \sim 1_M$, there exist morphisms $f$ and $g$ such that $f = 1_M f = \mb{c}_{X}(m_e) $ and $\mb{c}_{M}(m_e) g = 1_M$. Hence $f g = 1_M$ and $f g f = f$. Since $f = \mb{c}_{X}(m_e) \in \cm$, we have $ g f \sim 1_{_{\mb{c}_{M}(e(X))}}$ and there exists a morphism $h$ such that $g f h = 1_{_{\mb{c}_{M}(e(X))}}$. Therefore $f$ is an isomorphism.
\end{proof}
	
\begin{notation} We write, $e' \hs{1}\ulto{e} m$, whenever in the unbroken commutative diagram,
	$$\xymatrix{ \cdot \ar[r]^{e} \ar @{} [dr] |{\hs{-6}///} \ar [d]_{e'} & \cdot \ar[d]^{m}\\
		\cdot \ar[r] \ar@ {.>}[ur]_{d} & \cdot}$$
	there exists a morphism $d$ such that $e = d e'$.
\end{notation}

 \begin{remark}
	$e' \hs{1}\ulto{e} m$ is equivalent to: if $ \rangle me\langle \hspace{2mm}\subseteq\hspace{2mm} \rangle e'\langle $, then $ \rangle e\langle \hspace{2mm} \subseteq \hspace{2mm} \rangle e'\langle $.
\end{remark}

 Now we can define the class ${}^{\ulte}\hs{-2}\cm$ as follows:

 ${}^{\ulte}\hs{-2}\cm:\stackrel{\text{\tiny
		def}}{=} \{ e'\in \ce\mid e'\hs1 \ulto{e} m,\hskip1mm \forall e\in \ce \hskip1mm \text{and} \hs1 \forall m \in \cm \}$.

 For a closure operator $\mb C$ consider the following property:

 \begin{definition}
	Suppose $\mc{M}$ is a class of morphisms that is closed under composition, $\cx$ has quasi right $\cm$-factorizations and $\mb C$ is a closure operator. We say $\mb C$ satisfies the property (QCD) if compositions of quasi $\mb C$-dense morphisms are quasi $\mb C$-dense.
\end{definition}

 \begin{theorem}\label{T:weakly hereditary left}
	Suppose that $\cx$ has quasi right $\cm$-factorizations and $\mb C$ is a closure operator. If $ \cm \su QM(\cx) $ is closed under composition and $\ce^{QC} \su {}^{\ultqd}\hs{-2}\cm$, then $\cx$ has quasi left $\ce^{QC}$-factorization structures.
\end{theorem}

 \begin{proof}
	Suppose that $f = m_f e$ is a quasi right $\cm$-factorization of $f$. We will show that this factorization is also a quasi left $\ce^{QC}-$factorization of $f$. By proposition ~\ref{P:left part is dense}, we have $e \in \ce^{QC}$. Given $e' \in \ce^{QC}$ such that $\ra f \la \su \ra e' \la$, since $\ce^{QC} \su {}^{\ultqd}\hs{-2}\cm$, $\ra e \la \su \ra e' \la$. Therefore $f = m_f e$ is a quasi left $\ce^{QC}-$factorization of $f$.
\end{proof}

 \begin{theorem}
	Suppose that $\cx$ has quasi right $\cm$-factorizations and $\mb C$ is a quasi weakly hereditary closure operator. If $ \cm \su QM(\cx) $ is closed under composition, $\ce^{QC} \su {}^{\ultqd}\hs{-2}\cm$ and $ \ce^{QC} \su QE(\cx) $, then $\cm^{QC} = \cm$.
\end{theorem}

 \begin{proof}
	We only need to prove that $ \cm \su \cm^{QC}$. Let $m$ be an element of $\cm$ and $j_m$ be its quasi $\mb C$-dense morphism. Since $1_M \in \ce^{QC} \su {}^{\ultqd}\hs{-2}\cm$, $\ra 1_M \la \su \ra j_m \la$. So we have $\mb{c}(m) \leqslant m$, because $ \ce^{QC} \su QE(\cx)$. Thus $m \in \cm^{QC}$.
\end{proof}

 \section{Quasi factorization structures}
In this section the notations $\ch^{\APLup}$ and ${}^{\APLdown}\ch $ are introduced and after studying some of their properties, the notion of quasi factorization structure in a category $\cx$ is given. We will see that weak factorization structures as defined in \cite{ahr} are quasi factorization structures, but the converse is not true as we will show by some examples. Finally we state the relation between a quasi factorization structure and a quasi idempotent and quasi weakly hereditary closure operator.

 \begin{notation}
	
	(a) Given $\hs{-1}\xymatrix{E \ar[r]^{e} & X}\hs{-1} \in \cx$ and
	$\hs{-1}\xymatrix{M \ar[r]^{m} & X}\hs{-1}\in \cx$, $e \APLdown m$
	means that in every commutative triangle,
	$$\xymatrix{E \ar[rr]^{u} \ar[dr]_{e} \ar @{} [dr] |{\hs{14}///} & & M \ar[dl]^{m}\\
		& X &} $$
	there exists $\hs{-1}\xymatrix{w:X \ar[r] & M}\hs{-1}$ such that $m w \sim
	1_X$. (Note that $e \APLdown m$ means that $ e \leqslant m \Rightarrow 1_X
	\leqslant m$.)
	
	(b) Given $\hs{-1}\xymatrix{M \ar[r]^{e} & E}\hs{-1} \in \cx$ and
	$\hs{-1}\xymatrix{M \ar[r]^{m} & X}\hs{-1}\in \cx$, $e \APLup m$
	means that in every commutative triangle,
	$$\xymatrix{ & M \ar[dl] _{e} \ar[dr] ^{m} \ar @{} [dr] |{\hs{-14}///}\\
		E \ar[rr]_{v} & & X}$$
	there exists $\hs{-1}\xymatrix{w:E \ar[r] & M}\hs{-1}$ such that $ m w \sim v$.
	
	Let $\ch$ be a class of morphisms. We denote by $\ch^{\APLup}$ the class of all morphisms $m$ with
	$$h \APLup m \hs{2} \text{ for all } \hs{2} h\in \ch$$
	and similarly, by ${}^{\APLdown}\ch $ the class of all morphisms $e$ with
	$$e \APLdown h \hs{2} \text{ for all } \hs{2} h\in \ch.$$
\end{notation}

 Saying $\ch$ has $\cx$-pushouts if the pushout of
each morphism in $\ch$ exists and is in $\ch$, we have:

 \begin{proposition}\label{P:left right triangle} For each classes $\ch$,
$\ch_1$ and $\ch_2$ we have:
\begin{itemize}
	\itemi If $\ch_1 \hs{1}\su \ch_2$, then ${}^{\APLdown}\ch_2 \hs{1}\su
	{}^{\APLdown}\ch_1$ and $\ch_2^{\APLup} \su \ch_1^{\APLup}$.
	
	\itemii $\text{Ret}(\cx) \su {}^{\APLdown}\ch$.	
	
	\itemiii If $\ch \su QE(\cx)$, then $\text{Sec}(\cx) \su \ch^{\APLup}$.
	
	\itemiv If $\ch$ has $\cx$-pullbacks, then ${}^{\APLdown}\ch$ closed under
	composition. Dually, if $\ch$ has $\cx$-pushouts, then $\ch^{\APLup}$ closed
	under composition.
\end{itemize}
\end{proposition}

 \begin{proof}
The proof of (i) and (ii) follows directly from the definition.

 (iii) Suppose that the following commutative diagram
$$\xymatrix{ & M \ar[dl] _{h} \ar[dr] ^{s} \ar @{} [dr] |{\hs{-14}///}\\
	H \ar[rr]_{v} & & X}$$
such that $s$ is a section and $h \in \ch$ is given. Thus, $h$ is a section and
since $h$ is a quasi epi, by \cref{l:quasi mono split epi is iso}(b) $h \in
\iso(\cx)$. Put $w = h^{-1}$ and so $s w = v$.

 (iv) Suppose that$\xymatrix{E \ar[r]^{e_2} & E_1 \ar[r]^{e_1} & X}$are
composable morphisms in ${}^{\APLdown}\ch$ and the following commutative
triangle is given
$$\xymatrix{E \ar[rr]^{u} \ar[dr]_{e_1 e_2} \ar @{} [dr] |{\hs{14}///} & & M
\ar[dl]^{h}\\
	& X &} $$
such that $h \in \ch$. Since $\ch$ has $\cx$-pullbacks, we have
	$$\xymatrix{E \ar@{} [ddr] |{_{\hs{-5}///}} \ar@{} [drr] |{_{\hs{5}///}}
	\ar@/_2pc/[ddr]_{e_2} \ar@/^2pc/[drr]^{u}
	\ar@{.>}[dr]|{^{\exists ! d}} & \\
	& N \ar [d]_{e_1^{-1}(h)} \ar@{} [dr] |{p.b.}
	\ar[r]^{\hs{3}e_1^*}
	& M \ar [d]^{h} & \\
	& E_1 \ar[r]_{e_1} & X } $$
Thus, there exists a morphism$\xymatrix{E_1 \ar[r]^{w_1} & N}$ such that
$e_1^{-1}(h) w_1 \sim 1_{E_1}$ and hence $e_1^{-1}(h) w_1 g = 1_{E_1}$ for some
morphism$\xymatrix{g: E_1 \ar[r] & E_1.}$ Therefore we have the following
commutative triangle
$$\xymatrix{E_1 \ar[rr]^{e_1^* w_1 g} \ar[dr]_{e_1} \ar @{} [dr] |{\hs{14}///}
& & M
	\ar[dl]^{h}\\
	& X &} $$
and so there exists a morphism$\xymatrix{X \ar[r]^{w} & M}$such that $h w \sim
1_X$. This implies that $e_1 e_2 \in {}^{\APLdown}\ch$. The proof of the dual
is similar.
\end{proof}
\begin{proposition}\label{P:left triangle}
	Suppose that $\cx$ has quasi right $\cm$-factorizations and $\cm\su QM(\cx)$. If $\mb C$ is a closure operator, then $\ce^{QC} \su {}^{\APLdown}(\cm^{QC})$.
\end{proposition}

 \begin{proof}
	Suppose that $m u = e$,	where $e \in \ce^{ QC}$ and $m \in \cm^{ QC}$.
	Consider the quasi right $\cm$-factorization of $e$ as$$\xymatrix{E
	\ar[r]^{\hs{-4}e} & X = E \ar[r]^{\hs{2}e_1} & e(E) \ar[r]^{m_e} & X.}$$
	Since $e \in \ce^{ QC}$, by Proposition~\ref{P:closure of quasi dense is
	iso} we have $\mb{c}_{X}(m_e)$ is an isomorphism and so $\mb{c}_{X}(m_e)
	\in \cm^{ QC}$. We have $m_e = \mb{c}_{X}(m_e) j$ and so $e = \mb{c}_{X}(m_e) (j e_1)$. If $e \leqslant n$ with $n$ quasi closed, then $m_e \leqslant n$. Therefore $\mb{c}_X(m_e) \leqslant \mb{c}_X(n) \sim n$. Thus the factorization $e = \mb{c}_{X}(m_e) (j e_1)$ is a quasi right $\cm^{ QC}$-factorization of $e$. So we have the following commutative diagram.
	$$\xymatrix{E \ar[r]^{u} \ar[d]_{j e_1} & M \ar[d]^{m}\\
		\mb{c}_{X}(e(E)) \ar@<1ex> @{.>}[ur]^{w_1} \ar[r]_{\hs{2} \mb{c}_{X}(m_e)} \ar @{} [ur] |{\hs{8}///} &X} $$
	Put $w :\stackrel{\text{\tiny def}}= w_1 \mb{c}_{X}(m_e)^{-1}$, so $m w = 1_X$. Therefore $\ce^{QC} \su {}^{\APLdown}(\cm^{QC})$.
\end{proof}

 \begin{proposition}\label{P:right triangle}
	Suppose that $\cx$ has quasi right $\cm$-factorizations and $\cm\su QM(\cx)$ is closed under composition. If $\mb C$ is a quasi weakly hereditary closure operator such that $\ce^{ QC} \su QE(\cx)$ and $\ce^{QC} \su {}^{\ultqd}\hs{-2}\cm$, then $\cm^{QC} \su (\ce^{QC})^{\APLup}$.
\end{proposition}

 \begin{proof}
	Suppose that $m= v e$, where $e \in \ce^{ QC}$ and $m \in \cm^{QC}$.
	First we claim that$\xymatrix{M \ar[r]^{\hs{-5}m} & X =
	M \ar[r]^{\hs{1}j_m} & \mb{c}_{X}(M) \ar[rr]^{\mb{c}_{X}(m)} & &X}$is a
	quasi right $\cm$-factorization of $m$. For this reason suppose that the
	unbroken
	commutative diagram,
	$$\xymatrix{M \ar[r]^{u} \ar[d]_{j_m} \ar @{} [dr] |{\hs{9}///} & N \ar[d]^{n}\\
		\mb{c}_{X}(M) \ar[r]_{ \hs{4}\mb{c}_{X}(m)} \ar@{.>}[ur]^{w} & X} $$
	with $n \in \cm$, is given. By Theorem~\ref{T:weakly hereditary left} the quasi right $\cm$-factorization,
	$$\xymatrix{M \ar[r]^{\hs{-5}u} & N = M \ar[r]^{\hs{4}e_1} & E \ar[r]^{m_u} & N}$$
	of $u$ is a quasi left $\ce^{QC}$-factorization of $u$ and so $e_1 \in \ce^{QC}$. Now consider the following unbroken commutative diagram.
	$$\xymatrix{M \ar[r]^{e_1} \ar[d]_{j_m} \ar @{} [dr] |{\hs{-8}///} & E \ar[d]^{n m_u}\\
		\mb{c}_{X}(M) \ar[r]_{ \hs{4}\mb{c}_{X}(m)} \ar@{.>}[ur]_{d_1} & X} $$
	Since $\mb{C}$ is a quasi weakly hereditary closure operator, $j_m \in
	\ce^{QC}$ and since $\cm$ is closed under composition, we have $n m_u \in
	\cm$. Thus $\ce^{QC} \su {}^{\ultqd}\hs{-2}\cm$ implies that there exists a
	morphism $\hs{-1}\xymatrix{d_1: \mb{c}_{X}(M) \ar[r] & E}\hs{-1}$ such that
	$d_1 j_m = e_1$. Now define $w' :\stackrel{\text{\tiny def}}= m_u d_1$, so
	we have $n w' j_m = n m_u d_1 j_m = n m_u e_1 = \mb{c}_{X}(m) j_m$.
	Since $j_m \in \ce^{QC}$ and $\ce^{ QC} \su QE(\cx)$, we have $ n w' \sim \mb{c}_{X}(m)$. Therefore there exists a morphism $w''$ such that $\mb{c}_{X}(m) = n w' w''$. Define $w :\stackrel{\text{\tiny def}}= w' w''$. Thus $\mb{c}_{X}(m) \leqslant n$ and the claim is proved. Therefore by Theorem~\ref{T:weakly hereditary left} the factorization $m = \mb{c}_{X}(m) j_m$ is a quasi left $\ce^{ QC}$-factorization of $m$. So we have the following diagram,
	$$\xymatrix{M \ar[r]^{e} \ar[d]_{j_m} \ar @{} [dr] |{\hs{-7}///} & E \ar@<1ex> @{.>}[dl]^{d'} \ar[d]^{v}\\
		\mb{c}_{X}(M) \ar[r]_{ \hs{2}\mb{c}_{X}(m)} & X} $$
	such that $j_m = d' e$. Since $m \in \cm^{QC}$, we have $\mb{c}_{X}(m) \sim m$. Thus there exists $\hs{-1}\xymatrix{g: \mb{c}_{X}(M) \ar[r] & M}\hs{-1}$ such that $\mb{c}_{X}(m) = m g$ and hence $m g j_m = m$. Put $d :\stackrel{\text{\tiny def}}= g d'$. Thus $m d = m g d' = \mb{c}_{X}(m) d'$ and so $m d e = \mb{c}_{X}(m) d' e = v e$. Since $\ce^{ QC}\su QE(\cx)$, we have $ m d \sim v$.
\end{proof}

\begin{proposition}\label{P:quasi left is quasi dense}
	Suppose that $\cx$ has quasi right $\cm$-factorizations. Then ${}^{\APLdown}\cm \su \ce^{QC}$.
\end{proposition}

 \begin{proof}
	Let $e \in {}^{\APLdown}\cm$. Since $e \APLdown m_e$, $1 \leqslant m_e$. Thus $1 \leqslant \mb{c}(m_e)$.
\end{proof}

 Let $\ce$ and $\cm$ be classes of morphisms in $\cx$. We say that $\ce$ and
$\cm$ are closed under composition with isomorphisms

 (i) if $\alpha \in
\text{Iso}(\cx)$ and $e \in \ce$, and $\alpha e$ exists, then $\alpha e \in
\ce$;

 (ii) if $\alpha \in \text{Iso}(\cx)$ and $m \in \cm$, and $m \alpha$ exists,
then $m \alpha \in \cm$.

 \begin{proposition}\label{quasi mono and factorization}
	Let $\ce$ and $\cm$ be classes of morphisms in $\cx$ and for each
	morphism $f$ there exist $m \in \cm$ and $e \in \ce$ such that $f = m e$
	and $\cm \su QM(\cx)$.
	\begin{itemize}
		\itemi If $\ce$ is closed under composition with
		isomorphisms, then ${}^{\APLdown}\cm \su \ce$.
		
		\itemii If $\ce^{\APLup} \su QM(\cx)$ and $\cm$ is closed under composition
		with isomorphisms, then $\ce^{\APLup} \su \cm$.
	\end{itemize}
\end{proposition}

 \begin{proof}
	Consider the factorization $\xymatrix{X \ar[r]^{\hs{-5}f} & Y = X
		\ar[r]^{\hs{4}e} & M \ar[r]^{m} & Y}$ of $f$ where $m \in \cm$ and $e
		\in \ce$.
	
	(i) If $f \in {}^{\APLdown}\cm$, then there
	exists$\xymatrix{w:Y \ar[r] & M}$such that $m w \sim 1_Y$. Thus, $m$ is a
	split epi. Since $m$ is a quasi mono, by \cref{l:quasi mono split epi is
	iso}(a) $m\in \iso(\cx)$ and hence $f \in \ce$.
	
	(ii) If $f \in \ce^{\APLup}$, then by \cref{p:quasi mono}(b), $e$ is a quasi
	mono and there 	exists$\xymatrix{w:M \ar[r] & X}$such that $f w \sim m$.
	Thus, $mew \sim m$. Since $m$ is a quasi mono, by \cref{p:quasi mono}(a) $e
	w \sim 1_M$ and hence $e$ is a split epi. Therefore by \cref{l:quasi mono
	split epi is iso}(a) we have $e \in \iso(\cx)$ and hence $f \in \cm$.
	\end{proof}

In the following definition $\cx$ need not have quasi right $\cm$-factorizations.

 \begin{definition}
	A {\it quasi factorization structure} in a category $\cx$ is a pair $(\ce, \cm)$ of classes of morphisms such that;
	
	(a) every morphism $f$ has a factorization as$$\xymatrix{X \ar[r]^{\hs{-4}f}
	& Y = X \ar[r]^{\hs{2}e} & M \ar[r]^m & 
	Y,}$$ where $e\in \ce$ and $m \in
	\cm$;
	
	(b) $\ce \subseteq {}^{\APLdown}\cm$ and $\cm \subseteq \ce^{\APLup}$.
\end{definition}

 \begin{remark} (i)	If $(\ce, \cm)$ is a weak factorization structure in a category $\cx$, then
it is a quasi factorization structure.

 Suppose that $(\ce, \cm)$ is a quasi factorization structure in $\cx$ and $\cm
\su QM(\cx)$. By \cref{quasi mono and factorization} we have:

 (ii) $\ce \cap \cm \su \iso(\cx)$. To show this, let $f \in \ce \cap \cm$
and $$\xymatrix{X \ar[r]^{\hs{-4}f}
	& Y = X \ar[r]^{\hs{2}e} & M \ar[r]^m & Y}$$
where $e\in \ce$ and $m \in
\cm$. Since $f \in \ce \subseteq {}^{\APLdown}\cm$, there exists$\xymatrix{Y
\ar[r]^{w} & M}$ such that $m w \sim 1_Y$ and hence $m \in \iso(\cx)$. On the other
hand, $f \in \cm \subseteq \ce^{\APLup}$ and so there exists$\xymatrix{w':M
\ar[r] & X}$such that $f w' \sim m \sim 1_M$. Therefore $f$ is a split epi and
hence $f$ is an isomorphism.

 (iii) If $\ce$ is closed under composition with
isomorphisms, then $\ce = {}^{\APLdown}\cm$ and so by \cref{P:left right
triangle}(ii) $\text{Ret}(\cx)\su \ce$. Moreover, if $\cm$ has $\cx$-pullbacks,
then by \cref{P:left right triangle}(iv) $\ce$ is closed under composition.

 (iv) If $\ce^{\APLup} \su QM(\cx)$ and $\cm$ is closed under composition
with isomorphisms, then $\cm = \ce^{\APLup}$. Also if $\ce \su QE(\cx)$, then
by \cref{P:left right triangle}(iii) $\text{Sec}(\cx) \su \cm$. Moreover, if
$\ce$ has $\cx$-pushouts, then \cref{P:left right triangle}(iv) $\cm$ is closed
under composition
\end{remark}

 In the following example $(\ce, \cm)$ is a quasi factorization structure which
is not a weak factorization structure.

 \begin{example}
	(1) Let $\c$ be a closed model category whose objects are cofibrant. The
	pair
	$(\ce, \cm)$ of morphisms in $\c$, where $\ce$ is the class of weak
	equivalences
	and $\cm$ is the class of fibrations form a quasi factorization
	structure. To prove this, first note that every morphism $f$ in $\c$ has a
	factorization $f = p j$, where $j$ is a trivial cofibration and $p$ is a
	fibration, \cite[Definition 7.1.3]{h}. Now assume $\hs{-1}\xymatrix{E
	\ar[r]^{e} & X}\hs{-1} \in \ce$ and
	$\hs{-1}\xymatrix{M \ar[r]^{m} & X}\hs{-1}\in \cm$ and let $e = mu$ for
	some $u \in \c$. By \cite[Proposition 7.2.6]{h},$\xymatrix{E
	\ar[r]^{\hs{-4}e} & X = E \ar[r]^{\hs{2}i} & W \ar[r]^{p} & X,}$ where
	$i$ is a trivial cofibration and $p$ is a trivial fibration. Thus,
	\cite[Definition 7.1.3]{h} implies that there exists$\xymatrix{W \ar[r]^d &
	M}$such that $p =m d$. By \cite[Proposition 7.6.11]{h} there exists a
	morphism$\xymatrix{X \ar[r]^s & W}$such that $p s = 1_X$. Put $w = d s$, so
	$m w = 1_X$. Therefore $\ce \subseteq {}^{\APLdown}\cm$. Similarly, we can
	show that $\cm \subseteq \ce^{\APLup}$. Since $\ce \cap \cm \nsubseteq
	\iso{\cx}$, the system is not weak.
	
(2) As a special case of (1), in the category $\mathbf{Top}$, in which all the
objects are cofibrant, the collections $\ce$ of homotopy equivalences and $\cm$
of Serre fibrations form a quasi factorization structure.

(3) In Example~\ref{examp} (11) above, let $\ce=\{\hat{e_f}: \hat{f}\in
Set_\mathbb{P}\}$, where $e_f = \eta_{I_{f}} f'$. Then $(\ce, \cm)$ is a quasi
factorization structure. To show $(\ce, \cm)$ is not a weak factorization
structure, let $ X=\{x, x', x'' \} $ and consider the map $\hs{-1}\xymatrix{f:X
\ar[r]^{} & P(X)}\hs{-1} $ taking all the points to $ \{x\}$. Let $f =
\hat{m_{h}} \hat{k}$ be an $(\ce, \cm)$ factorization of $ \hat{f} $. As proved
in \cite{mh11} we can see the following commutative diagram has no diagonal.

	\centerline{\sqdg{X}{I_{g}}{I_{h}}{X}{\hat{u}}{\hat{k}}{\hat{m_{g}}}{\hat{v} \hat{m_{h}}}{///}}

(4)
	Let $\cx$ be a category with binary products in which projections are split epis. Let $\ce=Sec$ and $\cm=Ret$, where $Sec$ and $Ret$ denote the collection of all sections and split epis, respectively. Then $(\ce,\cm)$ is a quasi factorization structure. To show $(Sec,Ret)$ is not generally a weak factorization structure, let $Top-\{\emptyset\}$ be the full subcategory of $\mb{Top}$ consisting of the non-empty topological spaces and
	consider the following commutative diagram,

	\begin{center}
		\sqdg{\{0\}}{\{0,1,2\}}{\{0,1\}}{\{0,1\}}{u}{s}{r}{\tau}{///}
	\end{center}
	where $u$ sends 0 to 1, with codomain having $\{1\}$ open; $s$ is the inclusion with codomain having $\{1\}$ open; $r$ sends 0 to 0 and sends 1 and 2 to 1 with codomain having indiscrete topology; and $\tau$ is the twist map.
	It is easy to see that $s$ is a section and $r$ is a split epi. The square has no diagonal, since otherwise if $d$ is a diagonal, then $ds=u$ and $rd=\tau$. It follows that $d(0)=1$ and $d(1)=0$. Since $d^{-1}(\{1\})=\{0\}$, $d$ is not continuous.
	
(5)	Let $\cx$ be a category with coproducts,
	$$\ce=\{\xymatrix{A \ar[r]^{\hs{-3}\nu_{1}} & {A\amalg B} } \mid \nu_{1} \hs{1}\hbox{\small is the coproduct inclusion to the first factor}\}$$
	and $ \cm $ be any collection of split epis. Then $(\ce,\cm)$ is a quasi
	factorization structure. Since $\iso{(\cx)} \nsubseteq \ce$, the system is
	not weak.
	
(6) Let $\cx$ be an abelian category. Define $\ce=\{\xymatrix{\hs{-1}A \ar[r]^{ \hs{-4}\langle
		0,f\rangle} & {A\times B} } \mid \xymatrix{A \ar[r]^{f} & B} \in \cx\}$ and $\cm=\{\hs{-1}\xymatrix
{A\times B\ar[r]^{\hs{3}\pi_2} & B} \mid \pi_2 \hbox{ is the second projection}\}$.
Then $(\ce,\cm)$ is a quasi factorization structure. Since $\iso{(\cx)}
	\nsubseteq \cm$, the system is not weak.
\end{example}

 \begin{theorem}
	Suppose that $\cx$ has a quasi right $\cm$-factorizations such that $\cm$ is
	closed under composition, $\cm \su
	QM(\cx)$ and for each $f \in QM(\cx)$, $f$ is an isomorphism whenever $m_f$
	is an isomorphism. Then there is a class $\ce$ such that $(\ce, \cm)$ is a
	quasi factorization structure in $\cx$.
\end{theorem}

 \begin{proof}
	Let $\ce$ be the class ${}^{\APLdown}\cm$ and$\xymatrix{E
		\ar[r]^{\hs{-4}f} & Y = E \ar[r]^{\hs{2}e} & X \ar[r]^{m_f} & Y}$be a
		quasi right $\cm$-factorization of an arbitrary morphism $f$ in $\cx$.
		We show that $e \in \ce$. For this reason let the following commutative
		triangle 	
		$$\xymatrix{E \ar[rr]^{u} \ar[dr]_{e} \ar @{}
		[dr] |{\hs{14}///} & & M \ar[dl]^{m}\\
			& X &} $$
	such that $m \in \cm$ be given. Thus $f = m_f e = m_f m u$ and so there exists a
	morphism$\xymatrix{X \ar[r]^{d} & M}$such that $m_f m d = m_f$. Therefore
	$\la m d \ra = \la 1_M \ra$ and hence $m$ is a split epi. Since $m$ is a
	quasi mono, $m$ is an isomorphism. This implies that $e \in
	{}^{\APLdown}\cm = \ce$. Now we prove that $\cm \su \ce^{\APLup}$. Let $m \in
	\cm$ and the following commutative triangle
	$$\xymatrix{ & M \ar[dl] _{e} \ar[dr] ^{m} \ar @{} [dr] |{\hs{-14}///}\\
		E \ar[rr]_{v} & & X}$$
	such that $e \in \ce$ be given. Thus, by \cref{p:quasi mono}(b) we have $e \in
	QM(\cx)$. Let$\xymatrix{M \ar[r]^{\hs{-4}e} & E = M \ar[r]^{\hs{2}e'} & M_1
	\ar[r]^{m_e} & E}$be a
	quasi right $\cm$-factorization of $e$. Therefore there exists a
	morphism$\xymatrix{E \ar[r]^{w} & M_1}$ such that $m_e w \sim 1_E$ and
	hence $m_e$ is a split epi. Since $m_e \in QM(\cx)$, $m_e$ is an
	isomorphism and so by hypothesis we have $e$ is an isomorphism. Let
	$e^{-1}$ be the inverse of $e$, so $m e^{-1} = v$. Thus, $m \in
	\ce^{\APLup}$.
	\end{proof}
\begin{theorem}\label{qfs imples qr}
	Suppose that $(\ce, \cm)$ is a quasi factorization structure in $\cx$.
	\begin{itemize}
		\itema If $\cm$ has $\cx$-pullbacks, then $\cx$ has a quasi right $\cm$-factorization structure.
		
		\itemb If $\cm \su \mon(\cx)$ and $\ce$ has $\cx$-pushouts, then $\cx$ has a
		quasi left $\ce$-factorization structure.
	\end{itemize}
\end{theorem}

 \begin{proof}
	(a) Let$\xymatrix{f:X \ar[r] & Y}$be a morphism in $\cx$ and consider the quasi factorization of $f$ as$$\xymatrix{X
	\ar[r]^{\hs{-4}f} & Y = X \ar[r]^{\hs{2}e_f} & M \ar[r]^{m_f} & Y,}$$ where
	$e_f\in \ce$ and $m_f \in \cm$. Suppose the unbroken square,
	$$\xymatrix{X \ar[r]^{u} \ar[d]_{e_f} & N \ar[d]^{n}\\
		M \ar@<1ex> @{.>}[ur]^{d} \ar[r]_{\hs{2} m_f} \ar @{} [ur] |{\hs{6}///} & Y} $$
	is commutative, where $n \in \cm$. So there is a morphism $t$ such that the triangles in the following diagram commute.
	$$\xymatrix{X \ar@{} [ddr] |{_{\hs{-5}///}} \ar@{} [drr] |{_{\hs{5}///}}
		\ar@/_2pc/[ddr]_{e_f} \ar@/^2pc/[drr]^{u}
		\ar@{.>}[dr]|{^{\exists ! t}} & \\
		& m_{f}^{-1}(N) \ar [d]_{m_{f}^{-1}(n)} \ar@{} [dr] |{p.b.} \ar[r]^{\hs{3}m'}
		& N \ar [d]^{n} & \\
		& M \ar[r]_{m_f} & Y } $$
	Since $m_{f}^{-1}(n)\in \cm$ and $\ce \subseteq {}^{\APLdown}\cm$, there exists $\hs{-1}\xymatrix{w: M \ar[r] & m_{f}^{-1}(N)}\hs{-1}$ such that $m_{f}^{-1}(n) w \sim 1_M$. Thus there exists a morphism $\hs{-1}\xymatrix{\alpha: M \ar[r] & M}\hs{-1}$ such that $m_{f}^{-1}(n) w \alpha = 1_M$. Now define $d :\stackrel{\text{\tiny def}}{=} m' w \alpha$. So we have $n d = n m' w \alpha = m_f m_{f}^{-1}(n) w \alpha = m_f$.
	
	(b) Consider the quasi factorization of $f$ as$$\xymatrix{X
	\ar[r]^{\hs{-4}f} & Y = X \ar[r]^{\hs{2}e_f} & M \ar[r]^{m_f} & Y,}$$ where
	$e_f\in \ce$ and $m_f \in \cm$. Suppose the unbroken square,
	$$\xymatrix{X \ar[r]^{e} \ar[d]_{e_f} \ar @{} [dr] |{\hs{-4}///} & E \ar@<1ex> @{.>}[dl]^{d'} \ar[d]^{v}\\
		M \ar[r]_{m_f} & Y} $$
	is commutative, where $e \in \ce$. So there exists a morphism $t'$ such that the triangles in the following diagram commute.
	$$\xymatrix{ X \ar [d]_{e_f} \ar@{} [dr] |{p.o.} \ar[r]^{e} & E \ar [d]^{e'} \ar@/^2pc/[ddr]^{v} & \\
		M \ar@{} [dr] |{_{\hs{3}///}} \ar@/_2pc/[drr]_{m_f} \ar[r]_{e''} & E' \ar@{.>}[dr]|{^{\exists ! t'}} & \\
		&	& Y & }$$
	Since $e''\in \ce$ and $\cm \subseteq \ce^{\APLup}$, there exists $\hs{-1}\xymatrix{w': E' \ar[r] & M}\hs{-1}$ such that $m_f w' \sim t'$. Thus there exists $\hs{-1}\xymatrix{\beta: E' \ar[r] & E'}\hs{-1}$ such that $m_f w' \beta = t'$. Now define $d' :\stackrel{\text{\tiny def}}{=} w' \beta e'$. So we have $m_f w' \beta e' e = m_f w' \beta e'' e_f = t' e'' e_f = m_f e_f$. Therefore $d'e = e_f$.
\end{proof}

Calling $\cm$, {\it $\sim$-closed}, if whenever $m\in \cm$ and $f \sim m$, then $f\in \cm$, we have:

 \begin{corollary}\label{C:quasi}
	Suppose that $\cx$ has a quasi right $\cm$-factorization and $\cm \su QM(\cx)$ is closed under composition and is $\sim$-closed, the closure operator $\mb C$ is quasi weakly hereditary and quasi idempotent, and (QCD) holds for every $X\in \cx$. If $\ce^{ QC} \su QE(\cx)$ and $\ce^{QC} \su {}^{\ultqd}\hs{-2}\cm$, then $(\ce^{ QC}, \cm^{ QC})$ is a quasi factorization structure in $\cx$.
\end{corollary}

 \begin{proof}
	Let $f = m_f e_f$ be a quasi right $\cm$-factorization of $f$.
	By Theorems~\ref{T:closed right} and~\ref{T:weakly hereditary left}, the factorization $f = \mb{c}_{Y}(m_f) (j _{m_f} e_{f})$ is both quasi right $\cm^{ QC}$-factorization and quasi left $\ce^{QC}$-factorization of $f$. By Propositions~\ref{P:left triangle} and~\ref{P:right triangle} $(\ce^{ QC}, \cm^{ QC})$ is a quasi factorization structure in $\cx$.
\end{proof}

In \cite{dt}, it is proved that if the category $ \cx$ has $(\ce,\cm)$-factorization structures and $\mb C$ is a closure operator on $ \cx$, then $\mb C$ is idempotent and weakly hereditary iff $\cx$ has $(\ce^{ C}, \cm^{ C})$-factorizations.
In the following we prove a similar result under weaker conditions.

\begin{theorem}
	Suppose that $\cm \su QM(\cx)$ has $\cx$-pullbacks, is closed under composition and $(\ce, \cm)$ is a quasi factorization structure.
	\begin{itemize}
		\itemi Let $\cm$ be closed under composition with isomorphisms and the closure operator $\mb C$ be quasi weakly hereditary and quasi idempotent, and (QCD) holds for every $X\in \cx$. Then $(\ce^{ QC}, \cm^{ QC})$ is a quasi factorization structure if and only if $\cm^{ QC} \su ({}^{\APLdown}(\cm^{QC}))^{\APLup}$.
		
		\itemii If $(\ce^{ QC}, \cm^{ QC})$ is a quasi factorization structure, then $\cm^{ QC}$ is closed under composition and $\mb C$ satisfies the property (QCD).
	\end{itemize}
\end{theorem}

 \begin{proof}
(i)	Let $f = m e$ be a quasi factorization of $f$, where $e\in \ce$ and $m \in \cm$. By Propositions \ref{qfs imples qr} and \ref{P:quasi left is quasi dense}, $e \in \ce^{QC}$. Since $\mb C$ is quasi weakly hereditary, there exists a quasi $\mb C$-dense morphism $\hs{-1}\xymatrix{j_m : M \ar[r] & \mb{c}_{Y}(M)}\hs{-1}$ such that $\mb{c}_{Y}(m) j_m = m$. Also $$\hs{-1}\xymatrix{X \ar[r]^{\hs{-4}f} & Y = X \ar[r]^{j_m e} & \mb{c}_{Y}(M) \ar[r]^{\hs{2}\mb{c}_{Y}(m)} & Y,}\hs{-1}$$ where $\mb{c}_{Y}(m) \in \cm^{ QC}$. Put $d :\stackrel{\text{\tiny def}}= j_m e$. Since  $\ce^{ QC}$ is closed under composition, $d \in \ce^{ QC}$. Thus every morphism $f$ has a factorization such that its left part is in $\ce^{ QC}$ and its right part is in $\cm^{ QC}$. Now we show that $\ce^{QC} = {}^{\APLdown}(\cm^{QC})$. By \Cref{P:left triangle}, $\ce^{QC} \su {}^{\APLdown}(\cm^{QC})$. Let $h \in{}^{\APLdown}(\cm^{QC})$. Thus there exist morphisms $e_1 \in \ce^{QC}$ and $n_1 \in \cm^{QC}$ such that
$$\xymatrix{N \ar[r]^{\hs{-6}h} & E = N \ar[r]^{\hs{3}e_1} & h(N) \ar[r]^{n_1} & E.}$$
Thus there is a morphism $w_1$ as in the following diagram
	$$\xymatrix{N \ar[rr]^{\hs{-3}e_1} \ar[dr]_{h} & & h(N) \ar[dl]_{\hs{4} n_1}\\
	& E \ar@/^-1pc/ @{.>}[ur]_{w_1} & }$$
and hence $n_1 w_1 \sim 1_E$. This implies that $n_1$ is a split epi and since $n_1 \in \cm$, $n_1$ is an isomorphism. Thus by \Cref{QD is closed under comp. with iso}(a), $h \in \ce^{QC}$ and hence
\begin{equation}\label{eq main}
\ce^{QC} = {}^{\APLdown}(\cm^{QC}).
\end{equation}
Now, if $(\ce^{ QC}, \cm^{ QC})$ is a quasi factorization structure, since $\cm^{ QC} \su (\ce^{QC})^{\APLup}$ by equality (6), $\cm^{ QC} \su ({}^{\APLdown}(\cm^{QC}))^{\APLup}$. Converely, 
if $\cm^{ QC} \su ({}^{\APLdown}(\cm^{QC}))^{\APLup}$, then by equality (6), $\cm^{ QC} \su (\ce^{QC})^{\APLup}$ and so $(\ce^{ QC}, \cm^{ QC})$ is a quasi factorization structure.

 (ii) Consider morphisms$\xymatrix{M \ar[r]^{m_1} & X \ar[r]^{m_2} & Y}$ such that $m_1, m_2 \in \cm^{QC}$ and $$\xymatrix{M \ar[rr]^{\hs{-4}m_2 m_1} & & Y = M \ar[r]^{\hs{5}e} & N \ar[r]^{n} & Y}$$
be a $(\ce^{ QC}, \cm^{ QC})$-factorization of $m_2 m_1$. Consider the following pullback diagram
	$$\xymatrix{M \ar@{} [ddr] |{_{\hs{-5}///}} \ar@{} [drr] |{_{\hs{5}///}}
	\ar@/_2pc/[ddr]_{e} \ar@/^2pc/[drr]^{m_1}
	\ar@{.>}[dr]|{^{\exists ! s}} & \\
	& L \ar [d]_{m_2^*} \ar@{} [dr] |{p.b.} \ar[r]^{n^*}
	& X \ar [d]^{m_2} & \\
	& N \ar[r]_{n} & Y } $$
\Cref{mqc closed under compo} implies that $m_2^* \in \cm^{ QC}$ and since $\ce^{QC} \su {}^{\APLdown}(\cm^{QC})$, there exists$\xymatrix{s_1: N \ar[r] & L}$such that $m_2^* s_1 \sim 1_N$. Therefore $m_2^*$ is a split epi and since $m_2^* \in \cm$, $m_2^*$ is an isomorphism. Also $n^* (m_2^*)^{-1} e = n^* (m_2^*)^{-1} m_2^* s = n^* s = m_1$ and since $\cm^{ QC} \su (\ce^{QC})^{\APLup}$, we have the following diagram
$$\xymatrix{& M \ar[dl]^{e} \ar[dr]^{m_1} \\
	N \ar@/_-1pc/ @{.>}[ur]^{w_1'} \ar[rr]_{n^* (m_2^*)^{-1}} & & X}$$
such that $m_1 w_1' \sim n^* (m_2^*)^{-1}$ and hence there exists$\xymatrix{l: N \ar[r] & N}$such that $m_1 w_1' l = n^* (m_2^*)^{-1}$. Thus $m_2 m_1 (w_1' l) = m_2 n^* (m_2^*)^{-1} = n$ and hence $n \leq m_2 m_1$. Since $m_2 m_1 = n e$, $m_2 m_1 \leq n$ and so $m_2 m_1 \sim n$. Therefore $m_2 m_1 \in \cm^{QC}$. 

Now we prove that $\mb C$ satisfies (QCD). Let$\xymatrix{X \ar[r]^{e_1} & Y \ar[r]^{e_2} & Z}$ be morphisms such that $e_1, e_2 \in \ce^{QC}$. We show that $e_2 e_1 \in {}^{\APLdown}(\cm^{QC})$. Let$\xymatrix{X \ar[r]^{e_1} & Y \ar[r]^{\hs{-5}e_2} & Z = X \ar[r]^{\hs{2}u} & M \ar[r]^m & Z}$such that $m \in \cm^{ QC}$. Consider the following pullback diagram
	$$\xymatrix{X \ar@{} [ddr] |{_{\hs{-5}///}} \ar@{} [drr] |{_{\hs{5}///}}
	\ar@/_2pc/[ddr]_{e_1} \ar@/^2pc/[drr]^{u}
	\ar@{.>}[dr]|{^{\exists ! t}} & \\
	& K \ar [d]_{m^*} \ar@{} [dr] |{p.b.} \ar[r]^{k}
	& M \ar [d]^{m} & \\
	& Y \ar[r]_{e_2} & Z } $$
Thus $m^* \in \cm^{QC}$ and since $\ce^{QC} \su {}^{\APLdown}(\cm^{QC})$, there exists$\xymatrix{t_1: Y \ar[r] & K}$ such that $m^* t_1 \sim 1_Y$. Therefore $m^*$ is a split epi and since $m^* \in \cm$, $m^*$ is an isomorphism. Thus we have the following diagram
$$\xymatrix{Y \ar[rr]^{\hs{-3}k (m^*)^{-1}} \ar[dr]_{e_2} & & M \ar[dl]_{\hs{4} m}\\
	& Z \ar@/^-1pc/ @{.>}[ur]_{t_2} & }$$
such that $m t_2 \sim 1_Z$. Thus $e_2 e_1 \in {}^{\APLdown}(\cm^{QC})$. It is easy to see that $\cm^{QC}$ is closed under composition with isomorphisms and since $\cm^{QC}$ has $\cx$-pullbacks, by \Cref{qfs imples qr}(a), $\cx$ has quasi right $\cm^{QC}$-factorization. So by \Cref{QD is closed under comp. with iso}(a), $\ce^{QC}$ is closed under composition with isomorphisms. Thus by \Cref{quasi mono and factorization}(i), ${}^{\APLdown}(\cm^{QC}) \su \ce^{QC}$ and hence $e_2 e_1 \in \ce^{QC}$. Therefore $\mb C$ satisfies the property (QCD).
\end{proof}

 \begin{definition}\label{strong quasi mono}
	(a) A nonempty class $\cm$ is called a {\it codomain} if $m\in \cm$ and $\ra m \la \su \ra a \la$, yields $a\in \cm$.
	
	(b) A morphism $f$ is called a {\it strongly quasi mono}, whenever for every morphisms $a, b \in \cx$ if $f a = f b$, then $\la a \ra = \la b \ra$ and $\ra a \la\hs{1} = \hs{1} \ra b \la$.
\end{definition}

 \begin{example}
	Let $ \cc $ be a subcategory of $\mb{Set}$ and $X $ be an object in $\mb{Set}$ which is not in $\cc$.
	Now define the subcategory $\cd$ of $\mb{Set}$ to have $obj (\cc) \cup \{X\}$ as objects and for all $ A, B \in \cd$ , $ Hom_{\cd}(A,B) =$
	
	$$\begin{cases}
	Hom_{\cc}(A,B) & A, B \in \cc \cr Hom_{Set}(A,B) & A=X, B\in \cc \cr \emptyset & A \in \cc , B=X \cr \{\hs{-1} \xymatrix{X \ar[r]^f & X}\hs{-1} \mid f \hbox{ is a bijective map } \}& A=X , B=X
	\end{cases} $$
	
	For all $ A \in \cd$, the morphisms $\hs{-1} \xymatrix{f:X \ar[r] & A}\hs{-1}$ are strongly quasi monos in $\cd$.
\end{example}

 \begin{notation}
	The class of all strongly quasi monos is denoted by $SQM(\cx)$.
\end{notation}
Note that $\mon(\cx) \su SQM(\cx)$.
\begin{remark}
	(a) If $m \in \mon(\cx)$ and $\ra m \la\hs{1} \su \hs{1} \ra a \la$, then
	$a \in \mon(\cx)$. Therefore $\mon(\cx)$ is a codomain.
	
	(b) If $\cm$ is a codomain, then it is closed under composition with isomorphisms on the left.
	
	(c) If $\cm$ is a codomain and $\hs{-1}\xymatrix{M \ar[r]^m & X}\hs{-1} \in \cm$, then $1_M \in \cm$. Note that for each $\hs{-1}\xymatrix{X \ar[r]^f & Y}\hs{-1}$ since $f = f 1_X$, we have $\ra f \la\hs{1} \su \hs{1} \ra 1_X \la\hs{1}$. Therefore if
	$$\{M \mid M \text{ is a domain of an element of } \cm \} = Obj(\cx),$$
	then $\cm$ contains all the identities.
\end{remark}

 \begin{theorem}
	Suppose that $\cm \su SQM(\cx)$ and it is a codomain. If $\mb C$ is a closure operator such that $(\ce^{ QC}, \cm^{ QC})$ is a quasi factorization structure, then $\mb C$ is quasi weakly hereditary and quasi idempotent.
\end{theorem}

 \begin{proof}
	Consider an $(\ce^{ QC}, \cm^{ QC})$ quasi factorization
	structure$$\xymatrix{M \ar[r]^d & N \ar[r]^n & X}$$of $m = n d \in \cm$,
	where $d\in \ce^{ QC}$ and $n\in \cm^{ QC}$. So $m \leqslant n$ and hence
	$\mb{c}_X(m) \leqslant n$. Since $\ra m \la \hs{1} \su \hs{1}\ra d \la$,
	we have $d\in \cm$. Consider the following diagram.
	$$\xymatrix{& & M \ar@/_3pc/ [ddllrr]_{d} \ar[d]_{j_d} \ar[rr]^{1_M} & & M \ar[d]^{j_m} \ar@/^3pc/ [ddrrll]^{m} & &\\
		& & \mb{c}_N(M) \ar@{} [drr] |{///} \ar[d]_{\mb{c}_N(d)} \ar@ {.>}[rr]^w & & \mb{c}_X(M) \ar[d]^{\mb{c}_X(m)} & & \\
		& & N \ar[rr]_n & & X & &}$$
	Since $d\in \ce^{ QC} \cap \cm$, by Proposition~\ref{P:closure of quasi dense is iso} we have $\mb{c}_N(d)$ is an isomorphism. By \Cref{wdl} there exits$\xymatrix{w: \mb{c}_N(M) \ar[r] & \mb{c}_X(M)}$such that $n \mb{c}_N(d) = \mb{c}_X(m) w$. Since $\mb{c}_N(d)$ is an isomorphism, we have $n \leqslant \mb{c}_X(m)$. Therefore $\mb{c}_X(m) \sim n$, and hence $\mb{c}_X(\mb{c}_X(m)) \sim \mb{c}_X(m)$. Thus $\mb{C}$ is quasi idempotent. Since $\mb{c}_X(m) \leqslant n$, there exists $\hs{-1}\xymatrix{ \mb{c}_X(M) \ar[r]^{d'} & N}\hs{-1}$ such that $\mb{c}_X(m) = n d'$. Therefore $n = n d' w (\mb{c}_N(d))^{-1}$ and hence $\mb{c}_X(m) = \mb{c}_X(m) w(\mb{c}_N(d))^{-1} d'$. Thus $\ra d' w (\mb{c}_N(d))^{-1} \la \hs{1} = \hs{1} \ra 1_{N} \la$ and $\la w(\mb{c}_N(d))^{-1} d' \ra = \la 1_{_{\mb{c}_X(M)}} \ra$. These equivalences imply that $w(\mb{c}_N(d))^{-1}$ is an isomorphism. It follows that $w$ is an isomorphism. It is easy to see that $w j_d \leqslant j_m$. Thus by Proposition~\ref{QD is closed under comp. with iso}(b), we have $j_m \in \ce^{ QC}$.
\end{proof}

 \end{document}